\documentclass[12pt,reqno]{amsart}
\usepackage[headings]{fullpage}
\usepackage{amssymb,amsmath,amscd,bbm,tikz-cd}
\usepackage[all,cmtip]{xy}
\usepackage{url}
\usepackage[bookmarks=true,%
    colorlinks=true,%
    linkcolor=blue,
    citecolor=blue,%
    filecolor=blue,%
    menucolor=blue,%
    urlcolor=blue,%
    breaklinks=true]{hyperref}
\usepackage{slashed}
\usepackage{listings}
\usepackage{verbatim}
\usepackage{mathtools}
\usepackage[normalem]{ulem}   
\usepackage{graphicx}
\usepackage{texdraw}
\usepackage{caption}
\graphicspath{{figures/}}

\newcommand{\BA}{\mathbb{A}}
\newcommand{\BZ}{\mathbb{Z}}

\newcommand{\BN}{\mathbb{N}}

\newcommand{\calA}{\mathcal{A}}
\newcommand{\calC}{\mathcal{C}}

\newcommand{\calX}{\mathcal{X}}

\newcommand{\calT}{\mathcal{T}}
\newcommand{\sgn}{\mathrm{sgn}}
\newcommand{\LLT}{\mathrm{LLT}}
\newcommand{\Sk}{\mathrm{Sk}}

\newcommand{\SW}{\mathrm{SW}}
\newcommand{\inv}{\mathrm{inv}}
\newcommand{\qbinom}[2]{\genfrac{[}{]}{0pt}{}{#1}{#2}}

\usepackage[colorinlistoftodos]{todonotes}

\newtheorem{theorem}{Theorem}[section]
\newtheorem{corollary}[theorem]{Corollary}
\newtheorem{conjecture}[theorem]{Conjecture}
\newtheorem{theoremA}{Theorem}

\theoremstyle{definition}
\newtheorem{lemma}[theorem]{Lemma}
\newtheorem{definition}[theorem]{Definition}
\newtheorem{proposition}[theorem]{Proposition}
\newtheorem{example}[theorem]{Example}
\newtheorem{remark}[theorem]{Remark}

\begin{document}

\title[ ]{Chromatic symmetric functions for annular webs}

\author{Jaeseong Oh}
\address{Department of Mathematics \\
    Sungkyunkwan University \\
    Suwon, South Korea 
	\newline{\tt \url{https://sites.google.com/view/jaeseong-oh}}
}
\email{jaeseongoh@skku.edu}

\author{Seokbeom Yoon}
\address{Department of Mathematics, Center for Quantum Computing\\
	Chonnam National University \\
	Gwangju, South Korea 
	\newline{\tt \url{https://sites.google.com/view/seokbeom}}
}
\email{sbyoon15@gmail.com}


\date{\today}

\begin{abstract}
    We introduce a combinatorial definition of chromatic symmetric functions for annular webs. We prove their symmetry by constructing a web analogue of the Shareshian--Wachs involution and show that they coincide with the symmetric functions associated to annular webs via Turaev's isomorphism. We then derive explicit formulas for their hook Schur coefficients. 
    We also introduce web LLT functions, whose hook Schur coefficients admit positive Laurent-polynomial formulas. These formulas yield a combinatorial expression for the coefficients of the HOMFLY--PT polynomial of an annular web.
\end{abstract}

\maketitle

\section{Introduction}

    Turaev's isomorphism identifies the positive part of the skein algebra of the annulus with the ring of symmetric functions~\cite{Tur88}. This correspondence, in addition, associates a symmetric function with each annular web, with an essential circle of thickness $k$ corresponding to the elementary symmetric function $e_k$. Motivated by their categorification of the annular skein, Queffelec and Rose \cite{QR18} conjectured that every annular web is isomorphic to a direct sum of collections of thickened essential circles in the annular web/foam category. Under Turaev's isomorphism, this conjecture decategorifies to the following $e$-positivity conjecture.

    \begin{conjecture}[{\cite[Conjecture~5.4]{QR18}}]
        \label{conj:e-positivity}
        The symmetric function associated with every annular web is $e$-positive.
    \end{conjecture}
    
    A weaker version of this conjecture, namely Schur positivity, was proved by Gorsky and Wedrich~\cite[Corollary~1.4]{GW23} using Schur functors in the Karoubi-completed annular web category.

    \subsection{Overview}
    
    In this paper, we introduce a combinatorial model for these symmetric functions using colorings of annular webs.
    We present a brief summary of our approach and main results.

    For an annular web $w$, we define its \emph{chromatic symmetric function} by
    \[
        X_w[\mathbf x;q]:=\sum_{c\in\mathcal C(w)}q^{\inv_w(c)}x^c,
    \]
    where $\mathcal C(w)$ denotes the set of proper web colorings, $\inv_w(c)$ is a signed inversion statistic, and $x^c$ records the number of loops of each color. Precise definitions are given in Section~\ref{sec.chrom}. 
    
    We first prove the symmetry of $X_w$ by constructing a web analogue of the Shareshian--Wachs involution on $\calC(w)$, related to the construction in~\cite{KO26}, and identify $X_w$ with the symmetric function associated with $w$ under Turaev's isomorphism.

    \begin{theoremA}
    \label{thm.mainA}
        For each $i\geq1$, there is an involution on $\calC(w)$ that preserves $\inv_w(c)$ and interchanges the variables $x_i$ and $x_{i+1}$ in $x^c$. Consequently, $X_w$ is symmetric. Moreover, $X_w$ coincides with the symmetric function associated with $w$ under Turaev's isomorphism.
    \end{theoremA}

    We next present a combinatorial description of the hook Schur coefficients of $X_w$ in terms of $w$-tableaux, defined in Section~\ref{sec.Schur}. The proof combines the Jacobi--Trudi identity with a sign-reversing, weight-preserving involution on web arrays. This is analogous to the tableau formulas of Gasharov~\cite{Gas96} and Shareshian--Wachs~\cite{SW16}, although it is restricted to hook shapes.
    \begin{theoremA}
    \label{thm.mainB}
        Let $w$ be an annular web of degree $n$. For every hook partition $\lambda\vdash n$,
        \[
            [s_\lambda]X_w = \sum_{T\in\mathcal T_\lambda(w)} q^{\inv_w(T)} \in\mathbb Z_{\geq0}[q^{\pm1}],
        \]
        where $\mathcal T_\lambda(w)$ is the set of $w$-tableaux of
        shape $\lambda$.
    \end{theoremA}

Finally, we introduce a \emph{web LLT function} as a normalized generating function of \emph{arbitrary} colorings of labeled loop decompositions. This construction extends unicellular LLT polynomials, a special case of the symmetric functions
introduced by Lascoux, Leclerc, and Thibon~\cite{LLT97}.
Carlsson and Mellit~\cite{CM18} related unicellular LLT polynomials
to chromatic quasisymmetric functions by plethystic substitution.
Building on the thin-boundary identity of Kim and the first
author~\cite{KO26prep}, we extend this relation to annular webs:
\[
    \LLT_w[\mathbf x;q]
    =
    (q-q^{-1})^n
    X_w\!\left[\frac{\mathbf x}{q-q^{-1}};q\right],
    \qquad n=\deg(w).
\]
Together with the HOMFLY--PT specialization~\cite[Proposition~2.3]{GW23},
this gives
\[
    P_w(a,q)
    =\frac{\LLT_w[a-a^{-1};q]}{(q-q^{-1})^n},
\]
where $P_w$ is the (framed) HOMFLY--PT polynomial of $w$. 
By expressing the hook coefficients of $\LLT_w$ in terms of standard Young tableaux, we obtain the following combinatorial formula for the coefficients of $P_w(a,q)$. The notation appearing below that has not yet been defined is introduced in Section~\ref{sec:homfly-web}.

\begin{theoremA}
\label{thm.mainC}
Let $w$ be an annular web of degree $n$, and let $I$ be a
radial cut with initial thickness $\alpha$. Then for $0\leq k\leq n$,
the coefficient of $a^{2k-n}$ in $P_w$ is given by
\[
    [a^{2k-n}]P_w
    =
    \frac{(-1)^{n-k}
        \bigl(b_k(w,I;q)+b_{k+1}(w,I;q)\bigr)}
         {[\alpha]!(q-q^{-1})^n},
\]
where $b_0=b_{n+1}=0$ and for $1 \leq j \leq n$,
\[
    b_j(w,I;q)
    :=
    \sum_{\mathcal L\in\operatorname{SLD}(w,I)}
    \ \sum_{\tau\in\operatorname{SYT}(j,1^{n-j})}
        q^{\inv_{\mathcal L}(\tau)}.
\]
All other coefficients in $a$ vanish.
\end{theoremA}

\subsection{Organization}

In Section~\ref{sec.prelim}, we recall some basic definitions and facts about symmetric functions, annular links, and annular webs.
In Section~\ref{sec.chrom}, we introduce the chromatic symmetric function \(X_w\), prove its symmetry, and establish its compatibility with the skein-theoretic construction.
In Section~\ref{sec.Schur}, we prove the hook tableau formula.
In Section~\ref{sec:homfly-web}, we develop web LLT functions and derive the HOMFLY--PT coefficient formula.
Finally, in Section~\ref{sec.comparison}, we compare the chromatic symmetric functions associated with webs and graphs.

\section{Preliminaries}
\label{sec.prelim}

\subsection{Symmetric functions}
\label{sec.sym}

A \emph{weak composition} $\alpha$ of a positive integer $n$ is a finite sequence of nonnegative integers whose sum is $n$. We write $\alpha \models_0 n$, denote its $i$-th \emph{part} by $\alpha_i$, and denote its \emph{length} by $\ell(\alpha)$. We often identify $\alpha$ with its \emph{Young diagram}
\[
    \alpha=\{(i,j)\in \mathbb{Z}_{>0}\times \mathbb{Z}_{>0} \mid 1\leq i\leq \ell(\alpha),\ 1\leq j\leq \alpha_i\}.
\]
A \emph{composition} $\alpha$ of $n$, denoted by $\alpha \models n$, is a weak composition of $n$ whose parts are all positive.
For a composition $\alpha = (\alpha_1,\ldots, \alpha_r)$, we define its $q$-factorial by
$$
[\alpha]! = [\alpha_1]! \cdots [\alpha_r]! \quad \textrm{where} \quad [\alpha_i]!=[1]\cdots [\alpha_i], \quad [k] =\dfrac{q^k-q^{-k}}{q-q^{-1}}\,.
$$
A \emph{partition} $\lambda$ of $n$, denoted by $\lambda \vdash n$, is a composition of $n$ whose parts are weakly decreasing.
A partition is called a \emph{hook} if it is of the form $(k,1,\ldots,1)=(k,1^{n-k})$. 

Let $\Lambda_q$ be the algebra of symmetric functions in infinitely many variables $x_1,x_2,\dots$ over $\BZ[q^{\pm 1}]$. We use the standard notation for its usual bases: $e_\lambda$, $h_\lambda$, $m_\lambda$, and $s_\lambda$ denote the \emph{elementary, complete homogeneous, monomial,} and \emph{Schur symmetric functions}, respectively. In each case, the elements indexed by $\lambda \vdash n$ form a basis of the degree-$n$ part of $\Lambda_q$. The algebra $\Lambda_q$ is equipped with the \emph{Hall inner product} $\langle \cdot,\cdot \rangle$, which satisfies
\begin{equation*}
    \langle s_\lambda,s_\mu\rangle=\delta_{\lambda\mu}
    \quad\text{and}\quad
    \langle m_\lambda,h_\mu\rangle=\delta_{\lambda\mu}     
\end{equation*}
for any partitions $\lambda$ and $\mu$. Here $\delta_{\lambda\mu}$ denotes the Kronecker delta; that is, $\delta_{\lambda\mu}=1$ if $\lambda=\mu$, and $\delta_{\lambda\mu}=0$ otherwise.
    
A Schur symmetric function can be expressed as a signed sum of complete homogeneous symmetric functions via the Jacobi--Trudi formula:
\begin{equation}
\label{eq: Jacobi--Trudi}
    s_\lambda
    = \det \big( h_{\lambda_i+j-i} \big)_{i,j=1}^{\ell(\lambda)}
    = \sum_{\sigma \in S_{\ell(\lambda)}} \sgn(\sigma) \!\! \prod_{1 \leq j \leq \ell(\lambda)} \! \! h_{\lambda_{\sigma(j)}+j-\sigma(j)}
    = \sum_{\sigma \in S_{\ell(\lambda)}} \sgn(\sigma)\, h_{\sigma \cdot \lambda}
\end{equation}
where, in the last equality, 
\[
    {\sigma\cdot\lambda} := \big(\lambda_{\sigma(j)}+j-\sigma(j)\big)_{j=1}^{\ell(\lambda)} .
\]
Since $h_i=0$ for any $i<0$, the sums in~\eqref{eq: Jacobi--Trudi} may be regarded as being taken over those $\sigma$ for which $\lambda_{\sigma(j)}+j-\sigma(j)\geq 0$ for all $1\leq j\leq \ell(\lambda)$, or equivalently, for which $\sigma\cdot\lambda$ is a weak composition.

\subsection{Annular links}
\label{sec.skein}
    
Throughout the paper, we identify the annulus $\BA$ with  
\[
    \BA = [0,1]\times[0,1]/_\sim 
\]
where $(x,0)\sim (x,1)$ for all $0\leq x\leq 1$, and identify its first homology group $H_1(\BA;\mathbb{Z})$ with $\mathbb{Z}$ so that an upward-oriented vertical line represents $+1$.

By an \emph{annular link} $L$, we mean an oriented link in $\BA\times[0,1]$ represented by an annular diagram in which every strand is everywhere oriented upward. We refer to the homology class of $L$ in $H_1(\BA;\BZ)\simeq\BZ$ as the \emph{degree} of $L$.
    
The \emph{positive part of the skein algebra of the annulus}, denoted by $\Sk^{+}(\BA)$, is defined as the $\mathbb{Z}[q^{\pm1}]$-module generated by annular links, up to isotopy, modulo the skein relation as in  Figure~\ref{fig.skein}. The algebra structure is given by stacking links. The fact that stacking two copies of $\BA\times[0,1]$, one above the other, is isotopic to placing them side by side in the annulus, one inside and one outside, implies that $\Sk^+(\BA)$  is a commutative algebra.
\begin{figure}[!htbp]
    \centering
\begingroup%
  \makeatletter%
  \providecommand\color[2][]{%
    \errmessage{(Inkscape) Color is used for the text in Inkscape, but the package 'color.sty' is not loaded}%
    \renewcommand\color[2][]{}%
  }%
  \providecommand\transparent[1]{%
    \errmessage{(Inkscape) Transparency is used (non-zero) for the text in Inkscape, but the package 'transparent.sty' is not loaded}%
    \renewcommand\transparent[1]{}%
  }%
  \providecommand\rotatebox[2]{#2}%
  \newcommand*\fsize{\dimexpr\f@size pt\relax}%
  \newcommand*\lineheight[1]{\fontsize{\fsize}{#1\fsize}\selectfont}%
  \ifx\svgwidth\undefined%
    \setlength{\unitlength}{202.95969565bp}%
    \ifx\svgscale\undefined%
      \relax%
    \else%
      \setlength{\unitlength}{\unitlength * \real{\svgscale}}%
    \fi%
  \else%
    \setlength{\unitlength}{\svgwidth}%
  \fi%
  \global\let\svgwidth\undefined%
  \global\let\svgscale\undefined%
  \makeatother%
  \begin{picture}(1,0.22351376)%
    \lineheight{1}%
    \setlength\tabcolsep{0pt}%
    \put(0,0){\includegraphics[width=\unitlength,page=1]{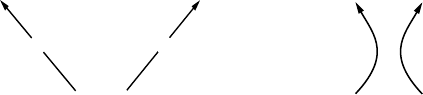}}%
    \put(0.21375331,0.10262678){\color[rgb]{0,0,0}\makebox(0,0)[lt]{\lineheight{1.25}\smash{\begin{tabular}[t]{l}$-$\end{tabular}}}}%
    \put(0.50008255,0.10346796){\color[rgb]{0,0,0}\makebox(0,0)[lt]{\lineheight{1.25}\smash{\begin{tabular}[t]{l}$=$\end{tabular}}}}%
    \put(0.59707314,0.10305537){\color[rgb]{0,0,0}\makebox(0,0)[lt]{\lineheight{1.25}\smash{\begin{tabular}[t]{l}$(q-q^{-1})$\end{tabular}}}}%
    \put(0,0){\includegraphics[width=\unitlength,page=2]{skein.pdf}}%
  \end{picture}%
\endgroup%

    \caption{The skein relation.}
    \label{fig.skein}
\end{figure}
    
It is well known, as first observed in \cite{Tur88}, that the algebra $\Sk^+(\BA)$
is isomorphic to the algebra $\Lambda_q$ of symmetric functions:
\begin{equation}
\label{eqn.Turaev}
    \Sk^{+}(\BA) \overset{\simeq}{\longrightarrow} \Lambda_q,   \quad L \longmapsto \calX_L \,.
\end{equation}
Explicitly, if $b$ is an $n$-strand braid, then its closure $\widehat{b}$ is mapped to
\begin{equation}
\label{eqn.turaev}
    \calX_{\widehat{b}} = \sum_{\lambda \vdash n} \operatorname{Tr}_{V_{\lambda}}(b)\, s_\lambda \,.
\end{equation}
Here $b$ is regarded as an element of the Hecke algebra $H_n$ of type $A_{n-1}$, and $V_{\lambda}$ denotes the irreducible representation of $H_n$ associated with the partition $\lambda$. In particular, this isomorphism preserves degrees. 

After extending the base ring $\BZ[q^{\pm1}]$ so that all $q$-integers are invertible, one can extend the correspondence~\eqref{eqn.Turaev} to \emph{thickened} annular links, where each component is assigned a positive integer. A component labeled by $k$ is interpreted as $k$ parallel copies of the component with the anti-$q$-symmetrizer inserted. The anti-$q$-symmetrizer on $k$ strands is given by
\[
    \frac{1}{\sum_{\sigma \in S_k} (-q)^{-2\ell(b_\sigma)}} \sum_{\sigma \in S_k} (-q)^{-\ell(b_\sigma)} b_\sigma 
\]
where $b_\sigma$ is a positive braid representative of $\sigma$ and $\ell(b_\sigma)$ denotes its length. For instance,  Figure~\ref{fig.sym2} shows the anti-$q$-symmetrizer on two strands. In this way, a thickened annular link can be regarded as a linear combination of annular links, and hence as an element of $\Sk^+(\BA)$.

\begin{figure}[!htbp]
    \centering
\begingroup%
  \makeatletter%
  \providecommand\color[2][]{%
    \errmessage{(Inkscape) Color is used for the text in Inkscape, but the package 'color.sty' is not loaded}%
    \renewcommand\color[2][]{}%
  }%
  \providecommand\transparent[1]{%
    \errmessage{(Inkscape) Transparency is used (non-zero) for the text in Inkscape, but the package 'transparent.sty' is not loaded}%
    \renewcommand\transparent[1]{}%
  }%
  \providecommand\rotatebox[2]{#2}%
  \newcommand*\fsize{\dimexpr\f@size pt\relax}%
  \newcommand*\lineheight[1]{\fontsize{\fsize}{#1\fsize}\selectfont}%
  \ifx\svgwidth\undefined%
    \setlength{\unitlength}{270.74027955bp}%
    \ifx\svgscale\undefined%
      \relax%
    \else%
      \setlength{\unitlength}{\unitlength * \real{\svgscale}}%
    \fi%
  \else%
    \setlength{\unitlength}{\svgwidth}%
  \fi%
  \global\let\svgwidth\undefined%
  \global\let\svgscale\undefined%
  \makeatother%
  \begin{picture}(1,0.16322293)%
    \lineheight{1}%
    \setlength\tabcolsep{0pt}%
    \put(0,0){\includegraphics[width=\unitlength,page=1]{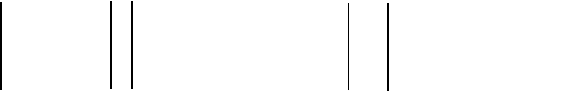}}%
    \put(0.01466498,0.1170255){\makebox(0,0)[lt]{\lineheight{1.25}\smash{\begin{tabular}[t]{l}$2$\end{tabular}}}}%
    \put(0.38707087,0.07769837){\makebox(0,0)[lt]{\lineheight{1.25}\smash{\begin{tabular}[t]{l}$\displaystyle\frac{1}{1+q^{-2}}$\end{tabular}}}}%
    \put(0.71631828,0.07672368){\makebox(0,0)[lt]{\lineheight{1.25}\smash{\begin{tabular}[t]{l}$-q^{-1}$\end{tabular}}}}%
    \put(0,0){\includegraphics[width=\unitlength,page=2]{sym2.pdf}}%
    \put(0.08275057,0.07579564){\makebox(0,0)[lt]{\lineheight{1.25}\smash{\begin{tabular}[t]{l}$=$\end{tabular}}}}%
    \put(0.31350038,0.07592233){\makebox(0,0)[lt]{\lineheight{1.25}\smash{\begin{tabular}[t]{l}$=$\end{tabular}}}}%
    \put(0,0){\includegraphics[width=\unitlength,page=3]{sym2.pdf}}%
  \end{picture}%
\endgroup%

    \caption{The anti-$q$-symmetrizer on two strands.}
    \label{fig.sym2}
\end{figure}

\subsection{Annular webs}
\label{sec.webs}

By an \emph{annular web} $w$, we mean a finite trivalent graph embedded in $\BA$, whose edges are all oriented upward, together with a \emph{thickness function} assigning a positive integer $t(e)$ to each edge $e$. We require that at every vertex the total incoming thickness equals the total outgoing thickness. Every vertex of $w$ has either two incoming edges and one outgoing edge, or one incoming edge and two outgoing edges. We call the former a \emph{merge} and the latter a \emph{split}. Since the total thickness is preserved at every vertex, the \emph{degree} of $w$ is defined by
\begin{equation*}
    \deg(w) := \sum_{e\cap I\neq \emptyset} t(e)
\end{equation*}
where $I$ is any radial segment in $\BA$ transverse to $w$.

An annular web $w$ may be regarded as an element of $\Sk^+(\BA)$ by extending the base ring so that all $q$-integers are invertible and inserting the anti-$q$-symmetrizer at every merge and split as illustrated in Figure~\ref{fig.symmetrizer}.
In this way, an annular web can be regarded as a linear combination of annular links. In particular, the correspondence~\eqref{eqn.Turaev} associates a symmetric function $\calX_w$ to an annular web $w$. Note that the degree of $\calX_w$ is equal to $\deg(w)$.
    
\begin{figure}[!htbp]
    \centering
    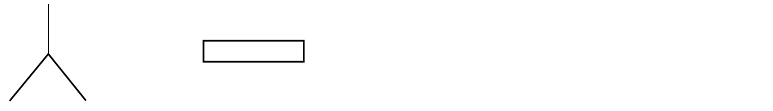
    \caption{Resolutions of a merge and a split.}
    \label{fig.symmetrizer}
\end{figure}

Conversely, an annular link can be expressed as a linear combination of annular webs. This can be done by resolving every crossing as in Figure~\ref{fig.skweb}, which follows by combining the relations in Figures~\ref{fig.sym2} and~\ref{fig.symmetrizer}. This crossing resolution is compatible with the skein relation in Figure~\ref{fig.skein}.

\begin{figure}[!htbp]
    \centering
    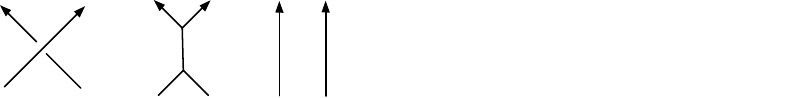
    \caption{Resolution of a crossing.}
    \label{fig.skweb}
\end{figure}

\section{Chromatic symmetric functions for annular webs}
\label{sec.chrom}

In this section, we give a combinatorial definition of the chromatic symmetric function of an annular web and prove the properties stated in the introduction.

\subsection{Definitions}

Let $w$ be an annular web. A  \emph{(proper) coloring} $c$ of $w$ is a function assigning to each edge $e$ a finite subset $c(e)\subset\mathbb{N}$ such that $|c(e)|=t(e)$, and at every vertex, the union of the colors on the incoming edges agrees with the union of the colors on the outgoing edges. For instance, at each merge, the union of the colors on the two incoming edges is equal to the set of colors on the outgoing edge, and hence the two sets of colors on the incoming edges are disjoint.
    
To define a chromatic symmetric function of $w$, we introduce an inversion statistic and an $x$-weight associated with a coloring $c$ as follows. 
    
At each merge $m$ of $w$, let $e_{\mathrm{L}}$ and $e_{\mathrm{R}}$ denote the left and right incoming edges, respectively. We set
$\inv_m(c) := \inv(c(e_\mathrm{R}),c(e_\mathrm{L}))$, where $\inv(A,B)$ for finite subsets $A,B \subseteq \BN$ is defined as 
\[
    \inv(A,B) := \sum_{a\in A,\, b\in B} \sgn(a-b),
\]
and  define the \emph{(signed) inversion} of the coloring $c$ by 
\begin{equation}
    \label{eq:def of inv}
    \inv_w(c)=\sum_{m : \textrm{merge}} \inv_m(c)  \,.
\end{equation}
Here the sum is taken over all merges of $w$.

For each $i \in \mathbb{N}$ let $\Gamma_i(c)$ be the union of edges $e$ of $w$ whose color $c(e)$ contains $i$. Then, by the coloring condition, $\Gamma_i(c)$ is a disjoint union of (positively oriented) circles. Denoting by $|\Gamma_i(c)|$ the number of its components, we define the \emph{$x$-weight} of $c$ as
\[
    x^c := \prod_{i \in \mathbb{N}} x_i^{|\Gamma_i(c)|}.
    \]
Note that its degree is equal to $\deg(w)$.
        
For instance, let $w$ be the annular web shown in the leftmost diagram of Figure~\ref{fig.SW}, with a coloring $c$ using two colors. If red and blue represent the colors $i$ and $i+1$, respectively, then its $x$-weight is $x^c=x_i^2x_{i+1}^4$, and $\inv_w(c)=0$: there are eight merges, four of which contribute $+1$ and the remaining four contribute $-1$.

\begin{definition} 
    \label{def.chromsym}
    For an annular web $w$, we define its \emph{chromatic symmetric function}  by
    \begin{equation}
        \label{eqn.chromsym}
        X_w[\mathbf x;q] := \sum_{c \in \mathcal{C}(w)} q^{\inv_w(c)} x^c 
    \end{equation}
    where $\calC(w)$ denotes the set of all colorings of $w$.
\end{definition}
It is immediate that $X_w$ is homogeneous of degree $\deg(w)$. However, its symmetry and its equality with the symmetric function $\calX_w$ introduced in Section~\ref{sec.webs} are not obvious. We establish these two properties in the next two subsections. In the process, we obtain Theorem~\ref{thm.mainA}.

\subsection{Shareshian--Wachs involution}
\label{sec.SW}
    
To prove the symmetry of $X_w$, it suffices to show that $X_w$ is invariant under interchanging two consecutive variables, say $x_i$ and $x_{i+1}$. For this purpose, we define an analogue of the Shareshian--Wachs involution \cite{SW16} with respect to the colors $i$ and $i+1$:
\begin{equation*}
    \SW_{i,i+1} : \calC(w) \rightarrow \calC(w).
\end{equation*}
This involution preserves the inversion weight and interchanges \(x_i\) and \(x_{i+1}\) in the \(x\)-weight, as shown in Theorem~\ref{thm.SW} below.
    
The construction of $\SW_{i,i+1}$ is as follows. For $c\in\calC(w)$, let $\Gamma_{i,i+1}(c)$ be the union of the edges $e$ of $w$ such that $c(e)$ contains exactly one of $i$ and $i+1$. We keep the orientation of the edges colored by $i+1$ and reverse the orientation of those colored by $i$, so that $\Gamma_{i,i+1}(c)$ becomes a disjoint union of oriented circles. Let $\Gamma_{i,i+1}^\ast(c)$ be the union of the components of $\Gamma_{i,i+1}(c)$ with nonzero winding number. We define $\SW_{i,i+1}(c)$ to be the coloring obtained from $c$ by swapping the colors $i$ and $i+1$ along the edges in $\Gamma_{i,i+1}^\ast(c)$. It follows from the construction that
\[
    \Gamma_{i,i+1}(c) = \Gamma_{i,i+1}(\SW_{i,i+1}(c)) \quad \text{and} \quad \Gamma^\ast_{i,i+1}(c) = \Gamma^\ast_{i,i+1}(\SW_{i,i+1}(c)),
\]
which implies that $\SW_{i,i+1}$ is an involution.

\begin{figure}[!htbp]
    \centering
\begingroup%
  \makeatletter%
  \providecommand\color[2][]{%
    \errmessage{(Inkscape) Color is used for the text in Inkscape, but the package 'color.sty' is not loaded}%
    \renewcommand\color[2][]{}%
  }%
  \providecommand\transparent[1]{%
    \errmessage{(Inkscape) Transparency is used (non-zero) for the text in Inkscape, but the package 'transparent.sty' is not loaded}%
    \renewcommand\transparent[1]{}%
  }%
  \providecommand\rotatebox[2]{#2}%
  \newcommand*\fsize{\dimexpr\f@size pt\relax}%
  \newcommand*\lineheight[1]{\fontsize{\fsize}{#1\fsize}\selectfont}%
  \ifx\svgwidth\undefined%
    \setlength{\unitlength}{409.40473902bp}%
    \ifx\svgscale\undefined%
      \relax%
    \else%
      \setlength{\unitlength}{\unitlength * \real{\svgscale}}%
    \fi%
  \else%
    \setlength{\unitlength}{\svgwidth}%
  \fi%
  \global\let\svgwidth\undefined%
  \global\let\svgscale\undefined%
  \makeatother%
  \begin{picture}(1,0.25139095)%
    \lineheight{1}%
    \setlength\tabcolsep{0pt}%
    \put(0,0){\includegraphics[width=\unitlength,page=1]{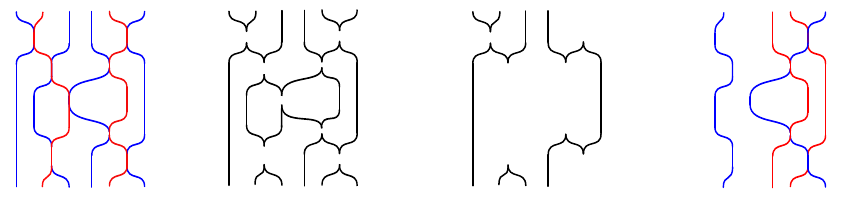}}%
    \put(0.0923634,-0.0060824){\color[rgb]{0,0,0}\makebox(0,0)[lt]{\lineheight{1.25}\smash{\begin{tabular}[t]{l}$c$\end{tabular}}}}%
    \put(0.58565136,-0.00656913){\color[rgb]{0,0,0}\makebox(0,0)[lt]{\lineheight{1.25}\smash{\begin{tabular}[t]{l}$\Gamma^\ast_{i,i+1}(c)$\end{tabular}}}}%
    \put(0.30601249,-0.00445921){\color[rgb]{0,0,0}\makebox(0,0)[lt]{\lineheight{1.25}\smash{\begin{tabular}[t]{l}$\Gamma_{i,i+1}(c)$\end{tabular}}}}%
    \put(0,0){\includegraphics[width=\unitlength,page=2]{SW.pdf}}%
    \put(0.8297986,-0.00481493){\color[rgb]{0,0,0}\makebox(0,0)[lt]{\lineheight{1.25}\smash{\begin{tabular}[t]{l}$SW_{i,i+1}(c)$\end{tabular}}}}%
  \end{picture}%
\endgroup%

    \caption{The SW involution: $\SW_{i,i+1}$ swaps the colors $i$ and $i+1$ on the edges in $\Gamma_{i,i+1}^\ast(c)$.}
    \label{fig.SW}
\end{figure}

\begin{theorem}
\label{thm.SW}
    The involution $\SW_{i,i+1}$ preserves the inversion weight while interchanging $x_i$ and $x_{i+1}$ in the $x$-weight. That is, for every $c \in \calC(w)$,
    $$\inv_w(\SW_{i,i+1}(c)) = \inv_w(c) \quad \text{and} \quad x^{\SW_{i,i+1}(c)} = x^c|_{x_i \leftrightarrow x_{i+1}}.$$
\end{theorem}
\begin{proof}
   Let $g$ be a component of $\Gamma_{i,i+1}(c)$. At each merge $m$ on $g$, the colors $i$ and $i+1$ occur on the two incoming edges. We set $\epsilon(m)=1$ if $i$ occurs on the right incoming edge, and $\epsilon(m)=-1$ otherwise. Similarly, at each split $s$ on $g$, the colors $i$ and $i+1$ occur on the two outgoing edges, and we define $\epsilon(s)=1$ if $i$ occurs on the right outgoing edge, and $\epsilon(s)=-1$ otherwise. 
   
    Recall that $\Gamma_i(c)$ is a disjoint union of oriented circles. We label its components by natural numbers from left to right, so that the first component of $\Gamma_i(c)$ is the leftmost one.
    Fix a split on $g$ as a starting point and follow $g$ along its orientation. Then we encounter splits and merges alternately. Suppose that, along $g$, we pass from a split $s$ to the next merge $m$. If $\epsilon(s)=\epsilon(m)$, then we remain on the same component of $\Gamma_i(c)$. If $\epsilon(s)=1$ and $\epsilon(m)=-1$, then we move to the component immediately to the left, while if $\epsilon(s)=-1$ and $\epsilon(m)=1$, then we move to the component immediately to the right. Thus the quantity $\epsilon(s)-\epsilon(m)$ records the change in the label of the component of $\Gamma_i(c)$ on which we lie. Since we return to the starting component after traversing $g$, the total change is zero. It follows that
    \begin{equation}
        \label{eqn.ms}
        \sum_{m\in g}\epsilon(m) = \sum_{s\in g}\epsilon(s)
    \end{equation}
    where the sums are taken over all merges and splits on $g$, respectively.

    Isotope $g$ in the annulus $\BA=[0,1]\times[0,1]/_\sim$ so that each split and merge becomes a cusp whose tangent vector is vertical.
    Then, at each split and merge, the exterior angle of $g$ is equal to $\epsilon(s)\pi$ and $\epsilon(m)\pi$, respectively. The turning tangent theorem implies that the sum of the exterior angles is determined by the rotation number. See, e.g., \cite[Section 4.5]{Carmo}. Precisely, we have
    \begin{equation}
        \label{eqn.turn}
        \left(\sum_{m\in g}\epsilon(m)+\sum_{s\in g}\epsilon(s)\right) \pi = 2\pi \left(o(g)-\omega(g)\right)
    \end{equation}
    where $\omega(g)$ denotes the winding number of $g$, and $o(g)$ denotes the sign of the orientation of $g$. Combining Equations~\eqref{eqn.ms} and~\eqref{eqn.turn}, we obtain
    \begin{equation}
        \label{eqn.s}
        \sum_{m\in g}\epsilon(m) = o(g)-\omega(g)  \,.
    \end{equation}
    Since $g$ is an oriented simple circle in the annulus, we deduce that
    \begin{equation*}
        \sum_{m\in g}\epsilon(m) = 
        \begin{cases}
            0 & \textrm{if } g \in \Gamma_{i,i+1}^\ast(c),\\
            o(g) & \textrm{if } g \notin \Gamma_{i,i+1}^\ast(c).
        \end{cases}    
    \end{equation*}
    When we apply $\SW_{i,i+1}$, the inversion of colors changes by the sum of $2\epsilon(m)$ over all merges on $\Gamma^\ast_{i,i+1}(c)$. Since $\Gamma^\ast_{i,i+1}(c)$ consists only of components with nonzero winding number, it follows that $\inv_w(\SW_{i,i+1}(c)) = \inv_w(c)$.
    
    On the other hand, it is clear that swapping the colors $i$ and $i+1$ along all the edges in $\Gamma_{i,i+1}(c)$ interchanges the roles of the variables $x_i$ and $x_{i+1}$. The map $\SW_{i,i+1}$ performs this swap only along the edges in $\Gamma^\ast_{i,i+1}(c)$. Every component of each $\Gamma_i(c)$ is a positively oriented essential circle, so $|\Gamma_i(c)|$ is equal to the total winding number of $\Gamma_i(c)$. The complement $\Gamma_{i,i+1}(c)\setminus \Gamma^\ast_{i,i+1}(c)$ consists of zero-winding components, and hence leaving those components unchanged does not affect the total winding numbers, equivalently the component counts, of $\Gamma_i(c)$ and $\Gamma_{i+1}(c)$. Therefore, $\SW_{i,i+1}$ interchanges the roles of $x_i$ and $x_{i+1}$ in the $x$-weight. This completes the proof.
\end{proof}

\begin{corollary}
\label{cor.sym}
    The chromatic symmetric function $X_w$ is symmetric.
\end{corollary}
\begin{proof}
    The involution $\SW_{i,i+1}$ allows us to partition $\calC(w)$ into three parts: $\calC_0$, consisting of colorings fixed by $\SW_{i,i+1}$, and two subsets $\calC_+$ and $\calC_-$, which are in bijection under $\SW_{i,i+1}$.
   Theorem~\ref{thm.SW} implies that, for each $c\in\mathcal{C}_0$, the monomial $q^{\inv_w(c)}x^c$ is symmetric with respect to $x_i$ and $x_{i+1}$. In addition,
    \begin{align*}
        \sum_{c\in\mathcal{C}_+} q^{\inv_w(c)}x^c + \sum_{c\in\mathcal{C}_-} q^{\inv_w(c)}x^c &=
        \sum_{c\in\mathcal{C}_+}
        (q^{\inv_w(c)}x^c + q^{\inv_w(\SW_{i,i+1}(c))}x^{\SW_{i,i+1}(c)}) \\
        &= \sum_{c\in\mathcal{C}_+}q^{\inv_w(c)} (x^c+x^c|_{x_i\leftrightarrow x_{i+1}}),
    \end{align*}
    which is also symmetric with respect to $x_i$ and $x_{i+1}$. Therefore, $X_w$ is invariant under the transposition of $x_i$ and $x_{i+1}$ for every $i$, and hence $X_w$ is symmetric.
\end{proof}

\begin{remark}
    A similar formulation of the involution $\SW_{i,i+1}$ can be found in \cite{KO26}, where the authors use the notion of a regular region, corresponding to $\Gamma^*_{i,i+1}$ in this paper. Their proof is inductive, whereas ours is direct.
\end{remark}

\subsection{Agreement with Turaev's isomorphism} 
\label{sec.agree}

Recall from Section~\ref{sec.webs} that any annular link can be expressed as a linear combination of annular webs by resolving each crossing as in Figure~\ref{fig.skweb}. Therefore, our definition of the chromatic symmetric function for annular webs extends to annular links. Moreover, since the crossing resolution in Figure~\ref{fig.skweb} satisfies the skein relation in Figure~\ref{fig.skein}, this extension defines a map
$$\Sk^+(\BA) \longrightarrow \Lambda_q, \qquad L \longmapsto X_L \,.$$
\begin{theorem} \label{thm.Turaev}
	The above map agrees with Turaev's isomorphism~\eqref{eqn.Turaev}.
	Therefore, for any annular web $w$, the two symmetric functions
	$X_w$ and $\calX_w$ coincide.
\end{theorem}
\begin{proof}
    Every annular link can be represented as the closure of a braid, so it suffices to consider an $n$-strand braid $b$.

    Let $\mathbf{1}_n$ denote the one-dimensional representation of the Hecke algebra $H_n$ on which every generator $T_i$ acts by $q$. For a partition $\mu=(\mu_1,\ldots,\mu_r)\vdash n$, let
    \[
        H_\mu = H_{\mu_1}\otimes\cdots\otimes H_{\mu_r} \subset H_n,
        \qquad
        \mathbf{1}_\mu = \mathbf{1}_{\mu_1}\otimes\cdots\otimes\mathbf{1}_{\mu_r}.
    \]
    We consider the permutation module $M^\mu:=H_n\otimes_{H_\mu}\mathbf{1}_\mu$ with a basis $\{v_{\mathbf a}\}$ indexed by words $\mathbf a=(a_1,\ldots,a_n)$ of content $\mu$. In this basis,
    \[
        T_i \cdot v_{(\ldots,k,l,\ldots)} =
        \begin{cases}
            v_{(\ldots,l,k,\ldots)}, & k<l,\\
            q\,v_{(\ldots,k,k,\ldots)}, & k=l,\\
            v_{(\ldots,l,k,\ldots)} +(q-q^{-1})v_{(\ldots,k,l,\ldots)}, & k>l.
        \end{cases}
    \]
    For a word $\mathbf a=(a_1,\ldots,a_n)$, set
    \[
        \inv(\mathbf a) = \#\{(i,j):i<j,\ a_i>a_j\}, \qquad
        \widetilde v_{\mathbf a} = (-q)^{-\operatorname{inv}(\mathbf a)}v_{\mathbf a}.
    \]
    Since interchanging adjacent entries $k<l$ increases the inversion number by one, while interchanging $k>l$ decreases it by one, the action in the basis $\{\widetilde v_{\mathbf a}\}$ is
    \[
        T_i \cdot \widetilde v_{(\ldots,k,l,\ldots)}
        =
        \begin{cases}
            -q\,\widetilde v_{(\ldots,l,k,\ldots)}, & k<l,\\
            q\,\widetilde v_{(\ldots,k,k,\ldots)}, & k=l,\\
            -q^{-1}\widetilde v_{(\ldots,l,k,\ldots)} +(q-q^{-1})\widetilde v_{(\ldots,k,l,\ldots)}, & k>l.
        \end{cases}
    \]
    This is exactly the rule governing the propagation of colors through a crossing under the resolution in Figure~\ref{fig.skweb}. Thus the coloring rule agrees with the action of $T_i$ on the basis $\{\widetilde v_{\mathbf a}\}$.

    When the braid $b$ is closed, a coloring contributes precisely when its outgoing word agrees with its incoming word. Therefore, summing over all words of content $\mu$ gives
    \[
        [x^\mu]X_{\widehat b} = \operatorname{Tr}_{M^\mu}(b), \qquad x^\mu=x_1^{\mu_1}\cdots x_r^{\mu_r}.
    \]
    Since $X_{\widehat b}$ is symmetric by Corollary~\ref{cor.sym}, this implies
    \[
        [m_\mu]X_{\widehat b} = \operatorname{Tr}_{M^\mu}(b).
    \]
    
    Using the monomial expansion of $s_\lambda$, we can rewrite
\eqref{eqn.turaev} as
    \[
        \begin{aligned}
            \mathcal X_{\widehat b} = \sum_{\lambda\vdash n} \operatorname{Tr}_{V_\lambda}(b)s_\lambda
            = \sum_{\lambda\vdash n} \operatorname{Tr}_{V_\lambda}(b) \sum_{\mu\vdash n}K_{\lambda\mu}m_\mu
            = \sum_{\mu\vdash n} \left( \sum_{\lambda\vdash n} K_{\lambda\mu}\operatorname{Tr}_{V_\lambda}(b) \right)m_\mu.
        \end{aligned}
    \]
    Here $K_{\lambda\mu}$ denotes the Kostka number.
    On the other hand, after extending scalars to $\mathbb Q(q)$, Young's rule gives
    \[
        M^\mu \simeq \bigoplus_{\lambda\vdash n} K_{\lambda\mu}V_\lambda, \quad \textrm{and hence} \quad
        \mathrm{Tr}_{M^\mu}(b) = \sum_{\lambda\vdash n} K_{\lambda\mu}\mathrm{Tr}_{V_\lambda}(b).
    \]
    Combining the above identities, we obtain
    \[
        X_{\widehat b} = \sum_{\mu\vdash n} \operatorname{Tr}_{M^\mu}(b)m_\mu = \mathcal X_{\widehat b}.
    \]
    Thus the map $L\mapsto X_L$ agrees with Turaev's isomorphism.
\end{proof}

By Theorem~\ref{thm.Turaev}, we identify $\calX_w$ with $X_w$ for any annular web $w$. It is known that this symmetric function satisfies the web relations and recovers the $\mathfrak{gl}_N$-polynomial \cite{MOY98}. We prove these properties directly from our combinatorial definition.

\begin{proposition}
    The chromatic symmetric function $X_w$ for positive annular webs  satisfies the web relations illustrated in Figure~\ref{fig.webrel}, as well as their mirror images.
    \begin{figure}[!htbp]
        \centering
        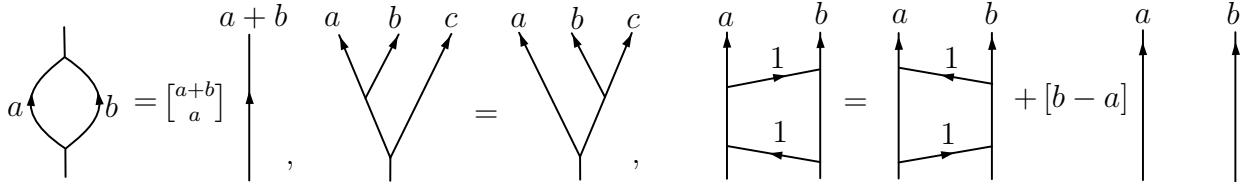
        \caption{Web relations.}
        \label{fig.webrel}
    \end{figure}
\end{proposition}

\begin{proof}
    Two elementary properties of $\inv$, namely $ \qbinom{|B|}{k}  = \sum_{A\subset B, |A|=k} q^{\operatorname{inv}(A,B\setminus A)}$ and
    \[
    \operatorname{inv}(A,B)+\operatorname{inv}(A\sqcup B,C)     =
    \operatorname{inv}(A,B\sqcup C)+\operatorname{inv}(B,C) \quad \textrm{for disjoint $A,B,C$},
    \]
    imply that $X_w$ satisfies the first and second relations and their mirrors.
    
    For the third relation, a comparison of all possible colorings shows that it suffices to check that
    \[
    \sum_{b\in B\setminus A} q^{\inv(A,b)+\inv(b,B)} - \sum_{a\in A\setminus B} q^{\inv(a,B)+\inv(A,a)} = [\,|B|-|A|\,].
    \]
    The contributions coming from elements of $A \cap B$ cancel in both exponents.  Thus it suffices to consider the elements of $A \Delta B =(A\setminus B)\cup(B\setminus A)$.
    Order these elements increasingly, $c_1<c_2<\cdots<c_m$,
    and define for each $k$
    \[
    h_{k-1} = |\{\,c_i\in B : i<k\,\}| - |\{\,c_i\in A : i<k\,\}|.
    \]  
    If $c_k \in B \setminus A$, then $h_k=h_{k-1}+1$ and $\inv(A,c_k)+\inv(c_k,B)=|A|-|B|+1+2h_{k-1}$. Thus
    \[
    q^{\operatorname{inv}(A,c_k)+\operatorname{inv}(c_k,B)}
    =  q^{|A|-|B|+1+2h_{k-1}}
    = \frac{q^{|A|-|B|+2h_k}-q^{|A|-|B|+2h_{k-1}}}{q-q^{-1}}.
    \]
    If $c_k \in A \setminus B$, then $h_k=h_{k-1}-1$ and $ \inv(c_k,B)+\inv(A,c_k) = |A|-|B|-1+2h_{k-1}.$ Thus
    \[
    -q^{\operatorname{inv}(c_k,B)+\operatorname{inv}(A,c_k)}
    = -q^{|A|-|B|-1+2h_{k-1}}
    = \frac{q^{|A|-|B|+2h_k}-q^{|A|-|B|+2h_{k-1}}}{q-q^{-1}}.
    \]    
    Consequently, both cases have the same telescoping form.
    Summing over all elements $c_k$ for $1 \leq k \leq m$, we obtain a telescoping sum:
    \begin{equation*}
        \sum_{b\in B} q^{\operatorname{inv}(A,b)+\operatorname{inv}(b,B)} -
        \sum_{a\in A} q^{\operatorname{inv}(a,B)+\operatorname{inv}(A,a)}
        = \frac{q^{|A|-|B|+2h_m}-q^{|A|-|B|+2h_0}}{q-q^{-1}}= [\,|B|-|A|\,].     
    \end{equation*}
    Note that $h_0=0$ and $h_m = |B|-|A|$. This completes the proof.
\end{proof}

\begin{proposition}
    The specialization of $X_w$ at $x_i = q^{2i-N-1}$ for $1 \leq i \leq N$ agrees with the $\mathfrak{gl}_N$-polynomial $\langle w \rangle_{\mathfrak{gl}_N}$ of an annular web $w$. That is,
    \begin{equation*}
        X_w(q^{-N+1},\ldots,q^{N-1}) = \langle w \rangle_{\mathfrak{gl}_N} \in \BZ[q^{\pm1}] \, .
    \end{equation*}
\end{proposition}
\begin{proof}
    Using Equation~\eqref{eqn.s}, one can rewrite the inversion of $c \in \calC(w)$ as
    \begin{equation*}
        \inv_w(c) = \sum_{i<j}  o(\Gamma_{i,j}(c))-\omega(\Gamma_{i,j}(c)) .
    \end{equation*}
    Here, $o(\Gamma_{i,j}(c)):=\sum_{g \in \Gamma_{i,j}(c)} o(g)$, and $\omega(\Gamma_{i,j}(c)):=\sum_{g \in \Gamma_{i,j}(c)} \omega(g)$. It follows that 
    \begin{align}
        X_w &= \sum_{c \in \calC(w)}
        q^{\sum_{i<j} o(\Gamma_{i,j}(c))-\omega(\Gamma_{i,j}(c))}
        \prod_i x_i^{|\Gamma_i(c)|} \,, \nonumber \\
        X_w(q^{-N+1},\ldots,q^{N-1}) & = \sum_{c \in \calC(w)} q^{\sum_{i<j} o(\Gamma_{i,j}(c)) - \omega (\Gamma_{i,j}(c))} \, q^{\sum_{i} {(2i-N-1)\, | \Gamma_i(c)} | }  \,. \label{eqn.spec}
    \end{align}
   Recall that $\Gamma_i(c)$ denotes the union of the edges $e$ of $w$ such that $c(e)$ contains $i$, and $\Gamma_{i,j}(c)$ denotes the union of the edges $e$ such that $c(e)$ contains exactly one of $i$ and $j$; $\Gamma_i(c)$ is a disjoint union of positively oriented circles, while $\Gamma_{i,j}(c)$ is a disjoint union of oriented circles by keeping the orientation of the edges colored by $j$ and reversing the orientation of those colored by $i$.
   
    Since $\Gamma_i(c)$ is a disjoint union of positively oriented circles, we have $|\Gamma_i(c)|=\omega(\Gamma_i(c))$. Combining this with the fact $\omega (\Gamma_{j}(c)) - \omega(\Gamma_{i}(c)) = \omega(\Gamma_{i,j}(c))$ for $i<j$, which follows from the definition of the winding number, we deduce that
    \[
        \sum_{i=1}^N (2i-N-1) |\Gamma_i(c)| = \sum_{1 \leq i <j \leq N} \omega(\Gamma_{i,j}(c)).
    \]
    This implies that the right-hand side of~\eqref{eqn.spec} simplifies to $\sum_c q^{\sum_{i<j} o(\Gamma_{i,j}(c))}$, which is  the definition of $\langle w \rangle_{\mathfrak{gl}_N}$; see e.g. \cite{Rob15}. This completes the proof.
\end{proof}

\section{Schur expansions of chromatic symmetric functions}
\label{sec.Schur}

\subsection{Loop decomposition}

Let $w$ be an annular web.  
A \emph{loop} in $w$ is a union of edges of $w$ that forms an oriented circle, with the orientations of the edges agreeing along the circle. For two loops $L_1$ and $L_2$, we write $L_1\prec L_2$ if $L_2$ lies in the right-hand region of the complement of $L_1$.
A \emph{loop decomposition} of $w$ is a collection of loops, with repetitions allowed, such that, for each edge $e$ of $w$, the number of loops containing $e$ is equal to the thickness of $e$. Note that the degree of $w$ is equal to the number of loops in any loop decomposition of $w$.

Let $n$ be the degree of $w$ and let $\alpha$ be a weak composition of $n$. A \emph{$w$-array} $T$ of shape $\alpha$ is an assignment of loops of $w$ to the boxes of the Young diagram of $\alpha$ such that the multiset of assigned loops forms a loop decomposition of $w$, and, in each row $i$, the entries satisfy
\[
    T_{i,1}\prec T_{i,2}\prec\cdots\prec T_{i,\alpha_i}.
\]
In particular, $T_{i,1},\ldots,T_{i,\alpha_i}$ are disjoint loops in $w$.

A $w$-array $T$ of shape $\alpha$ determines a coloring $c$ of $w$ by setting
\begin{equation*}
    c(e)=\{\, i \in \mathbb{N} \mid e \text{ is contained in a loop appearing in the $i$-th row of } T\,\}
\end{equation*}
for every edge $e$ of $w$. Note that the $x$-weight of this coloring $c$ coincides with $x^{\alpha}=\prod_i x_i^{\alpha_i}$.
Conversely, any coloring $c$ whose $x$-weight is $x^{\alpha}$ determines a $w$-array $T$ of shape $\alpha$: the $i$-th row of $T$ is obtained by arranging the loops in $\Gamma_i(c)$ according to the order $\prec$.
By abuse of notation, we write $\inv_w(T)=\inv_w(c)$. Then the definition~\eqref{eqn.chromsym} of the chromatic symmetric function can be rewritten as
\begin{equation*}
    X_w =
    \sum_{\alpha \models_0 n} \left( \sum_{T \in \calA_\alpha(w)} q^{\inv_w(T)} \right) x^{\alpha}
    =
    \sum_{\lambda \vdash n} \left( \sum_{T \in \calA_\lambda(w)} q^{\inv_w(T)} \right) m_\lambda 
\end{equation*}
where $\calA_\alpha(w)$ is the set of all $w$-arrays of shape $\alpha$. Note that the second equality follows from the symmetry of $X_w$.

\subsection{Invertibility and tableaux}

Let $T$ be a $w$-array of shape $\alpha \models_0 n$ and let $c$ be the corresponding coloring of $w$. We say that a cell $(i,j)$ with $i>1$ is \emph{admissible} if 
\begin{equation*}
    1\le j\le \min(\alpha_i,\alpha_{i-1}+1)\,.
\end{equation*}
Fix an admissible cell $(i,j)$ and let
\[
    A = \bigsqcup_{k\geq j} T_{i-1,k} \subset \Gamma_{i-1}(c), \qquad
    B = \bigsqcup_{k>j} T_{i,k} \subset \Gamma_i(c).
\]
Imitating the construction in Section~\ref{sec.SW}, let $\Gamma_{A,B}$ denote the union of the edges of $w$ that belong to exactly one of $A$ and $B$.
We keep the orientation on the $B$-edges and reverse the orientation on the $A$-edges so that $\Gamma_{A,B}$ is a disjoint union of oriented circles. Let $\Gamma^\ast_{A,B}$ be the union of the components of $\Gamma_{A,B}$ with nonzero winding number. Swapping the colors $i-1$ and $i$ along the edges in $\Gamma^\ast_{A,B}$ yields loop families $A'$ and $B'$, colored with $i$ and $i-1$, respectively, such that
\begin{equation*}
    A \cup B = A' \cup B', \qquad
    A \cap B = A' \cap B', \qquad
    |A| = |A'|, \qquad
    |B| = |B'|.
\end{equation*}
We denote by $\SW_{(i,j)}(T)$ the array obtained from $T$ by replacing $A$ and $B$ by $B'$ and $A'$, respectively, as in  Figure~\ref{fig.SWAB}. Note that $\SW_{(i,j)}(T)$ is of shape
\begin{equation*}
    \alpha'=(\ldots,\alpha_{i-2},\alpha_i-1,\alpha_{i-1}+1,\alpha_{i+1},\ldots) \models_0 n
\end{equation*}
which does not depend on $j$.

\begin{figure}[!htbp]
    \centering
    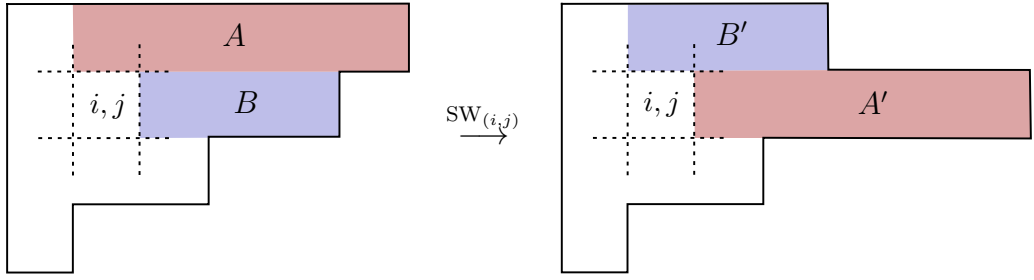
    \caption{The SW operation with respect to the cell $(i,j)$.}
    \label{fig.SWAB}
\end{figure}

Since no modifications are made outside $A'$ and $B'$, the resulting $\SW_{(i,j)}(T)$ need not be a $w$-array. More precisely, after the operation, the row-order condition may fail only at the cells $(i,j)$ and $(i-1,j-1)$, each with respect to its right neighbor.

\begin{definition}
    A $w$-array $T \in \calA_\alpha(w)$ is said to be \emph{invertible} at an admissible cell $(i,j)$ of $\alpha$ if $\SW_{(i,j)}(T)$ is again a $w$-array; equivalently, if
    \[
        \SW_{(i,j)}(T) \in \calA_{\alpha'}(w).
    \]
    We say that $T$ is \emph{invertible} if it is invertible at some admissible cell $(i,j)$ of $\alpha$. A \emph{$w$-tableau} is a $w$-array that is not invertible.
\end{definition}

It follows from the construction that if $T \in \calA_\alpha(w)$ is invertible at $(i,j)$, then $\SW_{(i,j)}(T)$ is also invertible at $(i,j)$. Moreover, $\SW_{(i,j)}(\SW_{(i,j)}(T))=T$.

\begin{example}
    Let $w$ be the annular web with loop decomposition $\{L_1,\ldots,L_4\}$ shown in Figure~\ref{fig.noninv}. Let $T$ be the $w$-array of shape $(2,2)$ whose first row is $L_2,L_4$ and whose second row is $L_1,L_3$. Applying $\SW_{(2,2)}$ to $T$, we have
    \[
        A=\{L_4\}, \quad B=\emptyset, \quad \Gamma^\ast_{A,B}=\{L_4\},  \quad  A'=\{L_4\}, \quad  B'=\emptyset.
    \]
    It follows that $\SW_{(2,2)}(T)$ consists of the same loops $L_1,\ldots, L_4$. However, since $L_3 \nprec L_4$, $\SW_{(2,2)}(T)$ is not a $w$-array.
    
    On the other hand, applying $\SW_{(2,1)}$ to $T$ changes the loop decomposition $\{L_1,\ldots, L_4\}$ into $\{L_1,L_2',L'_3,L'_4\}$ as illustrated in Figure~\ref{fig.noninv}. One checks that $\SW_{(2,1)}(T)$ is a $w$-array; that is,
$T$ is invertible at $(2,1)$.
    
    \begin{figure}[!htbp]
        \centering
        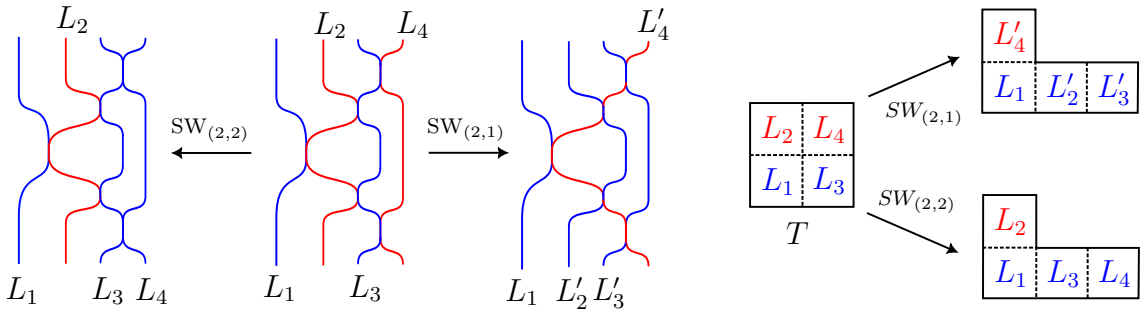
        \caption{An example of the SW operation on a $w$-array.}
        \label{fig.noninv}
    \end{figure}
\end{example}

The following proposition is useful when studying the hook Schur coefficients of the chromatic symmetric function $X_w$.

\begin{proposition}
\label{prop.exist}
Let $T$ be a $w$-array and let $(i,j)$ be an admissible cell of $T$. Suppose that either
\begin{equation}
	\label{eqn.condition}
    T_{i-1,j} = \emptyset \quad \text{or} \quad T_{i,j} \prec T_{i-1,j} \,.
\end{equation}
Then there is at least one cell among $(i,1),\ldots,(i,j)$
at which $T$ is invertible. 
\end{proposition}

\begin{proof}
	For an admissible cell $(i,r)$ with $1\leq r\leq j$, let
	\[
	\Gamma_r^\ast:=\Gamma_{A_r,B_r}^\ast,
	\qquad
	A_r=\bigsqcup_{k\geq r}T_{i-1,k},
	\qquad
	B_r=\bigsqcup_{k>r}T_{i,k}.
	\]
	To simplify the exposition, whenever a loop in $A_r$ and a loop
	in $B_r$ form a merge--split pair, we regard this pair as a
	transverse crossing, as illustrated in the figure below.  This
	convention does not change $\Gamma_r^\ast$, and hence the operation $\SW_{(i,r)}$
	is unchanged.  We first observe that components of $\Gamma_r^\ast$ and
	$\Gamma_{r'}^\ast$ for any $ 1 \leq r<r' \leq j$ cannot cross each other.  Indeed, under the
	convention above, this follows immediately from the local
	configurations at a transverse crossing.  
	\begin{figure}[!htbp]
		\centering
\begingroup%
  \makeatletter%
  \providecommand\color[2][]{%
    \errmessage{(Inkscape) Color is used for the text in Inkscape, but the package 'color.sty' is not loaded}%
    \renewcommand\color[2][]{}%
  }%
  \providecommand\transparent[1]{%
    \errmessage{(Inkscape) Transparency is used (non-zero) for the text in Inkscape, but the package 'transparent.sty' is not loaded}%
    \renewcommand\transparent[1]{}%
  }%
  \providecommand\rotatebox[2]{#2}%
  \newcommand*\fsize{\dimexpr\f@size pt\relax}%
  \newcommand*\lineheight[1]{\fontsize{\fsize}{#1\fsize}\selectfont}%
  \ifx\svgwidth\undefined%
    \setlength{\unitlength}{280.49300306bp}%
    \ifx\svgscale\undefined%
      \relax%
    \else%
      \setlength{\unitlength}{\unitlength * \real{\svgscale}}%
    \fi%
  \else%
    \setlength{\unitlength}{\svgwidth}%
  \fi%
  \global\let\svgwidth\undefined%
  \global\let\svgscale\undefined%
  \makeatother%
  \begin{picture}(1,0.20735716)%
    \lineheight{1}%
    \setlength\tabcolsep{0pt}%
    \put(0,0){\includegraphics[width=\unitlength,page=1]{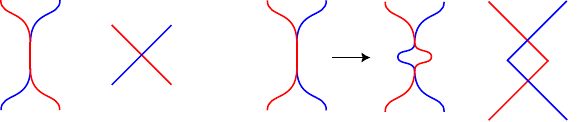}}%
    \put(0.1237076,0.10300969){\color[rgb]{0,0,0}\makebox(0,0)[lt]{\lineheight{1.25}\smash{\begin{tabular}[t]{l}$=$\end{tabular}}}}%
    \put(0.78839202,0.09759735){\color[rgb]{0,0,0}\makebox(0,0)[lt]{\lineheight{1.25}\smash{\begin{tabular}[t]{l}$=$\end{tabular}}}}%
  \end{picture}%
\endgroup%

	\end{figure}
	
	After applying $\SW_{(i,r)}$, the row-order condition can fail
	only at $(i,r)$ and $(i-1,r-1)$; let $I_r$ denote the row-order
	inequality at $(i,r)$ with respect to its right neighbor, and let
	$J_r$ denote the row-order inequality at $(i-1,r-1)$ with
	respect to its right neighbor.  As usual, an inequality is
	regarded as satisfied if the corresponding right neighbor does
	not exist.
	
	\smallskip
	\noindent \textit{Claim}. If $I_r$ holds and $J_r$ fails for
	$2\leq r\leq j$, then $I_{r-1}$ holds.

    \smallskip \noindent \textit{Proof of Claim.}
Choose a component $\alpha$ of $\Gamma_r^\ast$ responsible for the
failure of $J_r$. This means that some portion of $\alpha$ is colored by
$i$ and lies to the left of $T_{i-1,r-1}$. Since
$\alpha\subset\Gamma_r^\ast$, this $i$-colored portion of $\alpha$ lies
to the right of $T_{i,r}$. Hence, locally, the four curves occur from
left to right in the order
\begin{equation}
	\label{eqn.local}
	T_{i,r-1},\qquad T_{i,r},\qquad
	\alpha,\qquad T_{i-1,r-1}.
\end{equation}
If there is more than one such component $\alpha$, we choose the leftmost
one with respect to the local order~\eqref{eqn.local}.

Now suppose that $I_{r-1}$ fails. Choose a component $\beta$ of
$\Gamma_{r-1}^\ast$ responsible for this failure. This means that some portion of
$\beta$ is colored by $i-1$ and lies to the left of $T_{i,r-1}$.
Thus, locally, the three curves occur from left to right in the order
\begin{equation}
	\label{eqn.local2}
	\beta,\qquad T_{i,r-1},\qquad T_{i,r}.
\end{equation}

On the other hand, since $I_r$ holds, the $(i-1)$-colored portion of
$\Gamma_r^\ast$ lies to the right of $T_{i,r}$. Moreover, by the row-order
condition, the $i$-colored portion of $\Gamma_r^\ast$ also lies to the right
of $T_{i,r}$. Hence the whole of $\Gamma_r^\ast$ lies to the right of
$T_{i,r}$; in particular $\alpha$ lies to the right of $T_{i,r}$.
Combining the local order~\eqref{eqn.local2} with the fact that
$\alpha$ and $\beta$ cannot cross, we conclude that $\beta$
lies in the exterior region of $\alpha$.

Finally, it follows from $\beta\subset\Gamma_{r-1}^\ast$ that $\beta$
cannot lie to the left of both $T_{i,r}$ and $T_{i-1,r-1}$. Thus, in
the local order~\eqref{eqn.local}, $\beta$ must appear to the right of
$\alpha$, possibly with $\beta=T_{i-1,r-1}$. This contradicts the fact
that $\beta$ lies in the exterior region of $\alpha$. This proves the claim.

We now prove the proposition. The assumption~\eqref{eqn.condition}
implies that $I_j$ holds. Suppose, toward a contradiction, that $T$
is not invertible at any of the cells $(i,1),\ldots,(i,j)$.
Let
\[
r=\min\{\,1\leq s\leq j \mid I_s \text{ holds}\,\}.
\]
This is well-defined since $I_j$ holds. We claim that $r=1$.
Indeed, if $r\geq2$, then, since $T$ is not invertible at $(i,r)$
and $I_r$ holds, $J_r$ must fail. The Claim then implies that
$I_{r-1}$ holds, contradicting the minimality of $r$.
Thus $r=1$, so $I_1$ holds. Since $J_1$ holds automatically, $T$ is invertible
at $(i,1)$, a contradiction. Therefore, $T$ is invertible at at least
one of the cells $(i,1),\ldots,(i,j)$.
\end{proof}

\subsection{Hook Schur coefficients}

We devote this subsection to proving Theorem~\ref{thm.mainB}, which gives the hook Schur coefficients of the chromatic symmetric function $X_w$.

Let $w$ be an annular web of degree $n$, and let $\lambda$ be a hook partition of $n$. Recall that the Jacobi--Trudi identity and the Hall inner product give
\begin{align}
\label{eq:schur-coeff-jt}
    [s_\lambda]X_w
    =
    \langle X_w,s_\lambda\rangle
    & =
    \sum_{\sigma\in S_{\ell(\lambda)}}\sgn(\sigma)
    \langle X_w,h_{\sigma\cdot\lambda}\rangle                          \nonumber\\
    & =
    \sum_{\sigma\in S_{\ell(\lambda)}}\sgn(\sigma)
    [x^{\sigma\cdot\lambda}]X_w                                        \nonumber\\
    & =
    \sum_{\sigma\in S_{\ell(\lambda)}}\sgn(\sigma)
    \sum_{T\in\calA_{\sigma\cdot\lambda}(w)} q^{\inv_w(T)} \,,
\end{align}
where only those $\sigma$ for which $\sigma\cdot\lambda$ is a weak composition contribute.

\begin{lemma}
\label{lem:hook-composition}
Let $\lambda$ be a hook partition and let $\sigma \in S_{\ell(\lambda)}$. If $\sigma\cdot\lambda$ is a weak composition, then, for every $2\le i\le \ell(\lambda)$,
\[
    (\sigma\cdot\lambda)_i>1
    \quad\Longrightarrow\quad
    (\sigma\cdot\lambda)_{i-1}=0.
\]
In other words, every row of $\sigma\cdot\lambda$ other than the first that contains more than one box
is immediately preceded by an empty row.
\end{lemma}

\begin{proof}
Set $\mu=\sigma\cdot\lambda$ and suppose that $\mu_i>1$. If \(\sigma(i)=1\), then \(\sigma(i)<i\) since \(i\ge2\). Otherwise, \(\sigma(i)>1\), so \(\lambda_{\sigma(i)}=1\), and
$$ \mu_i=1+i-\sigma(i)>1 $$
implies \(\sigma(i)<i\).
Then $\sigma(i-1)\le i-1$: this is immediate if $\sigma(i-1)=1$, while if $\sigma(i-1)>1$, then $0<\mu_{i-1}=1+(i-1)-\sigma(i-1)$.
Similarly, for every $j \le i-2$, the nonnegativity of $\mu_j$ implies $\sigma(j)\le j+1\le i-1$ whenever $\sigma(j)>1$, and the same inequality is automatic when $\sigma(j)=1$. Hence
\[
    \sigma(1),\ldots,\sigma(i)\in\{1,\ldots,i-1\},
\]
which is impossible. Therefore, $\mu_{i-1}=0$.
\end{proof}

Lemma~\ref{lem:hook-composition} implies that, for every $i>1$,
\[
    \min\bigl((\sigma\cdot\lambda)_i,(\sigma\cdot\lambda)_{i-1}+1\bigr)\le 1.
\]
Thus, by the definition of admissible cells, every admissible cell of a
$w$-array $T$ of shape $\sigma\cdot\lambda$ lies in the first column.
In addition, combining Lemma~\ref{lem:hook-composition} with Proposition~\ref{prop.exist}, we deduce that an admissible cell $(i,1)$ is
invertible if and only if $T_{i-1,1}=\emptyset$ or $T_{i,1}\prec T_{i-1,1}$.

We first determine which $w$-arrays
$T\in\calA_{\sigma\cdot\lambda}(w)$ are noninvertible. Suppose that
$\sigma\ne\mathrm{id}$. If every part of $\sigma\cdot\lambda$ is
positive, then by Lemma~\ref{lem:hook-composition}, $(\sigma\cdot\lambda)_i=1$ for all $2\le i \le \ell(\lambda)$ and thus $(\sigma\cdot\lambda)=\lambda$. This forces $\sigma=\mathrm{id}$, a contradiction. It follows that 
$\sigma\cdot\lambda$ has an empty row. On the other hand, its last row is never empty, since
\[
    (\sigma\cdot\lambda)_{\ell(\lambda)}
    =
    \lambda_{\sigma(\ell(\lambda))}
    +\ell(\lambda)-\sigma(\ell(\lambda))
    >0.
\]
Therefore, every $T\in\calA_{\sigma\cdot\lambda}(w)$ has a nonempty row
whose preceding row is empty, and $T$ is
invertible at the first cell of such a row. Hence every
$T\in\calA_{\sigma\cdot\lambda}(w)$ is invertible whenever
$\sigma\ne\mathrm{id}$.
It follows that the only noninvertible $w$-arrays occur when
$\sigma=\mathrm{id}$, in which case the shape is $\lambda$. By
definition, these are precisely the $w$-tableaux in
$\calT_\lambda(w)$.

It remains to cancel all invertible $w$-arrays in pairs. For this
purpose, it is convenient to retain the permutation indexing the
corresponding Jacobi--Trudi term. Consider a pair $(\sigma,T)$, where $\sigma\in S_{\ell(\lambda)}$ and
$T\in\calA_{\sigma\cdot\lambda}(w)$ is invertible, and set
\[
    I=I(T):=\max\{\,i\mid T\text{ is invertible at }(i,1)\,\}.
\]
Since $T$ is invertible at $(I,1)$, $T'=\SW_{(I,1)}(T)$ is again a $w$-array, and its shape is $\sigma'\cdot\lambda$, where $\sigma'=\sigma \circ (I-1, I)$. Note that $\sgn(\sigma')=-\sgn(\sigma)$.

As in Section~\ref{sec.SW}, the operation $\SW_{(I,1)}$ preserves the inversion statistic. Indeed,
$\SW_{(I,1)}$ exchanges only the two adjacent colors $I-1$ and $I$
along the components of $\Gamma^\ast_{A,B}$. Since
exchanging two adjacent colors does not change their relative order
with any other color, only the contribution from the pair of colors
$I-1$ and $I$ can change. The same calculation as in the proof of
Theorem~\ref{thm.SW}, applied to the components of
$\Gamma^*_{A,B}$, shows that each nonzero-winding component contributes
zero to this change. Hence
\[
    \inv_w(T')=\inv_w(T).
\]

We next verify that the choice of $I$ is unchanged by the operation. The
cell $(I,1)$ remains invertible, since $\SW_{(I,1)}$ is an involution on
the set of $w$-arrays invertible at $(I,1)$. Moreover, the operation leaves
every row below the $I$-th row unchanged, as well as the cell
$(I,1)$. Since every admissible cell lies in the first column,
For $i>I$, the invertibility
of $(i,1)$ depends only on whether the $(i-1)$-st row is empty and on
the two first-column entries $T_{i-1,1}$ and $T_{i,1}$. All of these
data are unchanged by $\SW_{(I,1)}$. Thus no cell below row $I$
becomes invertible, and therefore
\[
    I(T')=I(T).
\]
It follows that applying the same construction to $(\sigma',T')$
returns $(\sigma,T)$. Hence
\[
    (\sigma,T)\longmapsto
    \bigl(
        \sigma \circ (I-1,I), \ 
        \SW_{(I(T),1)}(T)
    \bigr)
\]
defines a sign-reversing, $q$-weight-preserving involution on all terms
of~\eqref{eq:schur-coeff-jt}, except for the pairs
$(\mathrm{id},T)$ with $T\in\calT_\lambda(w)$. Therefore, all invertible terms cancel in pairs, and
 \eqref{eq:schur-coeff-jt} reduces to
\[
    [s_\lambda]X_w = \sum_{T\in\calT_\lambda(w)} q^{\inv_w(T)}.
\]
This proves Theorem~\ref{thm.mainB}.

\begin{remark}
The above proof does not extend directly to partitions $\lambda$ that are not hook-shaped.
Indeed, we have an explicit example involving three $w$-arrays
$T_1,T_2,T_3$, where $T_1$ and $T_3$ are obtained from $T_2$ by applying
the $\SW$-operation at certain cells, and there are no further $\SW$-relations with the others.
This means that the cancellation strategy used in the proof does not account for all three arrays.
\end{remark}

\section{The HOMFLY--PT polynomial of annular webs}
\label{sec:homfly-web}

We introduce an LLT function for annular webs, establish its plethystic
relation with the chromatic symmetric function $X_w$, and use it to compute the
HOMFLY--PT polynomial for annular webs. 

\subsection{LLT functions for annular webs}

Let $w$ be an annular web of degree $n$.
Fix a radial segment $I$ transverse to $w$ and disjoint from its vertices.
Let $e_1,\ldots,e_r$ be the edges intersecting $I$, listed from left to right,
and put
\[
    t(I)=(t(e_1),\dots,t(e_r)).
\]
We call $t(I)$ the \emph{initial thickness} relative to $I$. Assign to each $e_j$ the consecutive block
\[
    B_j=\{ t(e_1)+\cdots+t(e_{j-1})+1,
          \ldots,t(e_1)+\cdots+t(e_{j})\},
\]
so that $B_1,\ldots, B_r$ form a partition of $[n]=\{1,\ldots,n\}$.

A \emph{standard labeled loop decomposition} of the pair $(w,I)$ is a map from the edge set of $w$ to the power set of $[n]$ 
\[
    \mathcal L:E(w)\longrightarrow \mathcal{P}({[n]})
\]
satisfying the following conditions:
\begin{enumerate}
 	\item $\mathcal L(e_j)=B_j$ for $1\leq j\leq r$;
    \item $|\mathcal L(e)|=t(e)$ for every $e\in E(w)$;
    \item at every merge or split, the label set on the single edge on one
          side is the disjoint union of the label sets on the two edges
          on the other side.
\end{enumerate}
For each $u\in[n]$, the edges whose label sets contain $u$ form an oriented
loop, and we denote this loop by $L_u$.

For a standard labeled loop decomposition $\mathcal{L}$, we define $\mathsf A_I(\mathcal L)$ to
be the multiset of \emph{ordered attacking pairs} obtained as follows:
\begin{enumerate}
    \item for each $B_j$ and each $u<v$ in $B_j$, include one pair $(u,v)$;
          these are the \emph{initial attacks};
    \item at each merge, include one pair $(u,v)$ whenever $L_u$ enters
          from the left and $L_v$ enters from the right.
\end{enumerate}
All occurrences are counted with multiplicity.

For a map $\sigma:[n]\to\mathbb Z_{>0}$, we regard $\sigma(u)$
as the \emph{color} of $L_u$. No properness condition is imposed. For an attacking pair $(u,v)\in\mathsf A_I(\mathcal L)$, compare
the colors and break ties by decreasing loop index. The contribution is
\begin{equation}
\label{eq:llt-attack-sign}
    \epsilon_\sigma(u,v)
    :=
    \begin{cases}
        +1, & \sigma(v)>\sigma(u)
              \text{ or }(\sigma(v)=\sigma(u)\text{ and }v<u),\\
        -1, & \text{otherwise}.
    \end{cases}
\end{equation}
Define
\begin{equation}
\label{eq:llt-inversion}
    \inv_{\mathcal L}(\sigma)
    :=\sum_{(u,v)\in\mathsf A_I(\mathcal L)}\epsilon_\sigma(u,v),
    \qquad
    x^\sigma:=\prod_{u=1}^n x_{\sigma(u)}.
\end{equation}
For an attacking pair $(u,v)$ arising at a merge with
$\sigma(u)\neq\sigma(v)$, the contribution is
$\operatorname{sgn}(\sigma(v)-\sigma(u))$,
as in the definition of $X_w$.

\begin{definition}
\label{def:web-llt}
The \emph{web LLT function of $w$ relative to $I$} is
\begin{equation*}
    \LLT_{w,I}[\mathbf x;q]
    :=\frac{1}{[t(I)]!}
      \sum_{\mathcal L\in\operatorname{SLD}(w,I)}
      \ \sum_{\sigma:[n]\to\mathbb Z_{>0}}
      q^{\inv_{\mathcal L}(\sigma)}x^\sigma,
\end{equation*}
where $\operatorname{SLD}(w,I)$ is the set of all standard labeled loop decompositions of $(w,I)$.
\end{definition}

It is immediate that the web LLT function is homogeneous of degree $n$ with coefficients in $\mathbb Q(q)$. Its symmetry and independence of the choice of radial cut \(I\) follow from Proposition~\ref{prop:web-llt-plethysm} below.

\subsection{The plethystic relation}

We use plethystic notation in the symmetric-function variables. A substitution $f[A;q]$ is
computed by writing $f$ in the power-sum basis and replacing each $p_m$ by
$p_m[A]$, extending multiplicatively and $\mathbb Q(q)$-linearly. The
substitutions needed below are determined by
\begin{equation}
\label{eq:plethystic-conventions}
    p_m\!\left[\frac{\mathbf x}{q-q^{-1}}\right]
      =\frac{p_m[\mathbf x]}{q^m-q^{-m}},
    \qquad
    p_m[a-a^{-1}]=a^m-a^{-m}.
\end{equation}

Carlsson and Mellit~\cite{CM18} established a plethystic relation between
unicellular LLT polynomials and chromatic quasisymmetric functions. We use
the following extension due to Kim and the first author~\cite{KO26prep}:
if an annular web $v$ of degree $n$ has a radial cut $J$ of initial
thickness $(1^n)$, then, in the signed convention used here,
\begin{equation}
\label{eq:thin-boundary-plethysm}
    \LLT_{v,J}[\mathbf x;q]
    =(q-q^{-1})^n
      X_v\!\left[\frac{\mathbf x}{q-q^{-1}};q\right].
\end{equation}
The case of general initial thickness reduces to the
thin-boundary case by the digon relation for $X_w$.

\begin{proposition}
\label{prop:web-llt-plethysm}
Let $w$ be an annular web of degree $n$. Then for any radial cut $I$,
\begin{equation*}
    \LLT_{w,I}[\mathbf x;q]
    =(q-q^{-1})^n
      X_w\!\left[\frac{\mathbf x}{q-q^{-1}};q\right].
\end{equation*}
Consequently, $\LLT_{w,I}$ is symmetric and independent of $I$; we write
it as $\LLT_w$.
\end{proposition}

\begin{proof}
For each edge of thickness $d$ meeting $I$, insert a split--merge web that
splits the edge into $d$ unit strands and then merges them back. Choose the
merge tree to combine the unit strands successively from left to right,
and choose the split tree to be its reverse. Denote the resulting web by
$\widetilde w$, and choose a cut $J$ through all the inserted unit strands, so that
its initial thickness is $(1^n)$.
Repeated digon removal gives
\begin{equation}
\label{eq:thin-resolution-chromatic}
    X_{\widetilde w}=[t(I)]!\,X_w.
\end{equation}

Label the unit strands along $J$ increasingly from left to right.
Restriction to the original web gives a bijection
\[
    \operatorname{SLD}(\widetilde w,J)
    \longrightarrow \operatorname{SLD}(w,I).
\]
Its inverse extends each label along the uniquely prescribed branch of the
inserted split and merge trees. Within $B_j$, the inserted merges produce
exactly one attack $(u,v)$ for every $u<v$. Thus these merge attacks are
precisely the initial attacks of $\mathcal L$, and the bijection preserves
the inversion statistic for every loop coloring. Since $J$ has no initial
attacks and its normalizing factor is $1$, we obtain
\begin{equation}
\label{eq:thin-resolution-llt}
    \LLT_{\widetilde w,J}=[t(I)]!\,\LLT_{w,I}.
\end{equation}
Applying \eqref{eq:thin-boundary-plethysm} to $\widetilde w$ and using
\eqref{eq:thin-resolution-chromatic}--\eqref{eq:thin-resolution-llt} gives
\[
    [t(I)]!\,\LLT_{w,I}
    =(q-q^{-1})^n
      X_{\widetilde w}\!\left[\frac{\mathbf x}{q-q^{-1}};q\right]
    =[t(I)]!(q-q^{-1})^n
      X_w\!\left[\frac{\mathbf x}{q-q^{-1}};q\right].
\]
Dividing by $[t(I)]!$ proves the assertion.
\end{proof}

\begin{remark}
\label{rem:alternative-llt-proof}
An alternative approach to
Proposition~\ref{prop:web-llt-plethysm} is to verify directly that the
normalized state sum defines a graded algebra map on the positive
annular skein over $\mathbb Q(q)$. This requires checking cut independence,
compatibility with the web and skein relations, and multiplicativity
under disjoint union.

The remaining calculation concerns an essential circle $C_d$ of
thickness $d$. Its unique standard labeled loop decomposition has only
the initial attacks $(u,v)$ with $u<v$. The signed $q$-multinomial
identity and the $q$-binomial theorem give
\[
    \LLT_{C_d}[\mathbf x;q]
    =
    \sum_{\lambda\vdash d}
    \frac{q^{-\sum_j\binom{\lambda_j}{2}}}
         {\prod_j[\lambda_j]!}\,m_\lambda[\mathbf x]
    =
    (q-q^{-1})^d
    e_d\!\left[\frac{\mathbf x}{q-q^{-1}}\right].
\]
Under Turaev's isomorphism, $C_d$ corresponds to $e_d$. Since the
elementary symmetric functions generate the ring of symmetric functions
and the degree-normalized plethystic substitution is multiplicative,
this calculation identifies the two algebra maps on generators and hence
proves Proposition~\ref{prop:web-llt-plethysm} without using the
thin-boundary result.
\end{remark}

\subsection{The HOMFLY--PT specialization}

We now turn to the relation with the HOMFLY--PT polynomial. We use the normalization in which the uncolored unknot has value $(a-a^{-1})/(q-q^{-1})$.

The rules in \eqref{eq:plethystic-conventions} give
\[
    p_m\!\left[\frac{a-a^{-1}}{q-q^{-1}}\right]
    =\frac{a^m-a^{-m}}{q^m-q^{-m}}.
\]
This defines an algebra homomorphism from the ring of symmetric functions
over $\mathbb Q(q)$ to $\mathbb Q(a,q)$. 

\begin{proposition}[{\cite[Proposition~2.3]{GW23}}]
\label{prop:homfly-specialization}
Let $L$ be an annular link,  $X_L$ be its image under
Turaev's isomorphism, and $P_L$ be its HOMFLY--PT polynomial. Then
\[
    P_L(a,q)
    =X_L\!\left[\frac{a-a^{-1}}{q-q^{-1}};q\right].
\]
\end{proposition}

Accordingly, for an annular web $w$ of degree $n$,
define its HOMFLY--PT polynomial by
\begin{equation}
\label{eq:homfly-from-llt}
    P_w(a,q)
    :=X_w\!\left[\frac{a-a^{-1}}{q-q^{-1}};q\right]=\frac{\LLT_w[a-a^{-1};q]}{(q-q^{-1})^n}.
\end{equation}
Here, the last equality follows from Proposition~\ref{prop:web-llt-plethysm}.

It is standard that the specialization $s_\lambda[a-a^{-1}]$ of the Schur function vanishes
unless $\lambda$ is a hook; see, for example,
\cite[Chapter I, Section 3, Example 23]{Mac95}. More precisely, for
$\lambda\vdash n$,
\begin{equation}
\label{eq:hook-schur-specialization}
    s_\lambda[a-a^{-1}]
    =
    \begin{cases}
        (-1)^{n-k}
        \bigl(a^{2k-n}-a^{2k-n-2}\bigr),
        & \text{if } \lambda=(k,1^{n-k}),\\
        0,
        & \text{otherwise}.
    \end{cases}
\end{equation}
Thus, only the hook Schur coefficients of $\LLT_w$ contribute to
\eqref{eq:homfly-from-llt}; we compute them next.

For a composition $\alpha=(\alpha_1,\ldots,\alpha_s)\models n$, let
$\calA_\alpha$ denote the set of fillings of the Young diagram of
$\alpha$ with $1,\ldots,n$, each appearing exactly once, such that the
entries increase from left to right in each row. We call an element of
$\calA_\alpha$ an \emph{array} of shape $\alpha$. Each $T\in\calA_\alpha$ determines a loop coloring
$\sigma_T:[n]\to \BZ_{>0}$ satisfying $x^{\sigma_T}=x^\alpha$ by
\[
    \sigma_T(u)=i
    \qquad\text{if $u$ lies in the $i$-th row of $T$}.
\]
By abuse of notation, we write $\inv_{\mathcal L}(T):=\inv_{\mathcal L}(\sigma_T)$. Then the definition of $\LLT_w$ gives
\begin{equation}
\label{eq:llt-array-coefficient}
    [t(I)]!\,[x^\alpha]\LLT_w
    =
    \sum_{\mathcal L\in\operatorname{SLD}(w,I)}
    \ \sum_{T\in\calA_\alpha}
        q^{\inv_{\mathcal L}(T)}.
\end{equation}

For $T\in\calA_\alpha$, define its \emph{decreasing row word}
$\operatorname{rw}(T)$ by reading each row from right to left, starting
with the first row. Write
\[
    \operatorname{rw}(T)=r_1r_2\cdots r_n.
\]
If $\tau_T\in S_n$ is defined by
\[
    \tau_T^{-1}=\operatorname{rw}(T),
\]
then $\tau_T$ is obtained by standardizing the row coloring
$\sigma_T$, breaking ties by decreasing label. By the tie-breaking
convention in \eqref{eq:llt-attack-sign}, every attack has the same
contribution for $\sigma_T$ and $\tau_T$. Hence
\begin{equation}
\label{eq:array-standardization}
    \inv_{\mathcal L}(T)=\inv_{\mathcal L}(\tau_T).
\end{equation}
In particular, the inversion weight depends only on the decreasing
row word.

\begin{theorem}
\label{thm:web-llt-hook}
Let $w$ be an annular web of degree $n$, and let $I$ be a radial cut with initial thickness $t(I)$. 
Then the hook Schur coefficient of $\LLT_w$ is given by
\begin{equation}
\label{eq:llt-hook-coeff}
    [s_{(k,1^{n-k})}]\LLT_w
    =
    \frac{1}{[t(I)]!}
    \sum_{\mathcal L\in\operatorname{SLD}(w,I)}
    \ \sum_{T\in\operatorname{SYT}(k,1^{n-k})}
        q^{\inv_{\mathcal L}(T)},
\end{equation}
for $1\leq k\leq n$.
\end{theorem}

\begin{proof}
Put $\lambda=(k,1^{n-k})$ and $r=\ell(\lambda)$. We extend
$\calA_\alpha$ to weak compositions by allowing empty rows.
As in Section~\ref{sec.Schur}, the Jacobi--Trudi identity and \eqref{eq:llt-array-coefficient} give
\begin{equation}
\label{eq:llt-hook-signed-arrays}
    [t(I)]!\,[s_\lambda]\LLT_w
    =
    \sum_{\mathcal L\in\operatorname{SLD}(w,I)}
    \ \sum_{\sigma\in S_r}
        \sgn(\sigma)
        \sum_{T\in\calA_{\sigma\cdot\lambda}}
            q^{\inv_{\mathcal L}(T)},
\end{equation}
where terms for which $\sigma\cdot\lambda$ has a negative part
are omitted.

Fix $\mathcal L$ and $(\sigma,T)$, and put
$\alpha=\sigma\cdot\lambda$. By Lemma~\ref{lem:hook-composition},
every row $i>1$ with more than one entry has an empty predecessor.
A nonempty row $i>1$ is \emph{switchable} if row $i-1$ is empty
or $T_{i,1}<T_{i-1,1}$.

Choose the largest switchable $i$, if one exists. If row $i-1$
is empty, move all but the first entry of row $i$ into it.
Otherwise, row $i$ has a single entry; move the entries of
row $i-1$ after that entry. The resulting array $T'$ has
increasing rows and the same decreasing row word as $T$.
Thus \eqref{eq:array-standardization} gives
\[
    \inv_{\mathcal L}(T')=\inv_{\mathcal L}(T).
\]
The affected row lengths change by
\[
    (\alpha_{i-1},\alpha_i)
    \longmapsto
    (\alpha_i-1,\alpha_{i-1}+1).
\]
Hence $T'$ has shape $\sigma'\cdot\lambda$, where
$\sigma'=\sigma\circ(i-1,i)$ and
$\sgn(\sigma')=-\sgn(\sigma)$.

Row $i$ remains switchable with the same first entry, and all
lower rows are unchanged. Thus $i$ is still the largest
switchable index, and applying the switch again recovers $T$.
Therefore $(\sigma,T)\mapsto(\sigma',T')$ is a sign-reversing,
weight-preserving involution.

Finally, the last row is nonempty, since
\[
    (\sigma\cdot\lambda)_r
    =
    \lambda_{\sigma(r)}+r-\sigma(r)>0.
\]
If an empty row occurs, the next nonempty row is switchable.
Thus an unpaired array has no empty rows.
Lemma~\ref{lem:hook-composition} then forces $\alpha=\lambda$
and hence $\sigma=\mathrm{id}$. On this shape, having no
switchable row means that the first column increases from
top to bottom. The unpaired arrays are therefore exactly
the standard Young tableaux of shape $\lambda$.

All surviving terms have positive sign, so
\eqref{eq:llt-hook-signed-arrays} reduces to
\[
    [t(I)]!\,[s_\lambda]\LLT_w
    =
    \sum_{\mathcal L\in\operatorname{SLD}(w,I)}
    \ \sum_{T\in\operatorname{SYT}(\lambda)}
        q^{\inv_{\mathcal L}(T)},
\]
which proves the theorem.
\end{proof}

Together with the definition in \eqref{eq:homfly-from-llt}, the hook Schur formula yields the following expression for the coefficients of the HOMFLY--PT polynomial, which is a reformulation of Theorem~\ref{thm.mainC}.

\begin{corollary}
\label{cor:homfly-hook-formula}
Let $w$ be an annular web of degree $n$, and let $I$ be a radial cut with initial thickness $t(I)$. 
Then, for $0\leq k\leq n$, the coefficient of $a^{2k-n}$ in $P_w$ is given by
\begin{equation}
\label{eq:homfly-hook-formula}
    [a^{2k-n}]P_w
    =
    \frac{(-1)^{n-k}}
         {[t(I)]!(q-q^{-1})^n}
    \sum_{\mathcal L\in\operatorname{SLD}(w,I)}
    \ \sum_{\substack{j\in\{k,k+1\}\\1\leq j\leq n}}
    \ \sum_{T\in\operatorname{SYT}(j,1^{n-j})}
        q^{\inv_{\mathcal L}(T)}.
\end{equation}
All other coefficients in $a$ vanish.
\end{corollary}

\begin{proof}
Substituting the Schur expansion of $\LLT_w$ into
\eqref{eq:homfly-from-llt} and using
\eqref{eq:hook-schur-specialization}, we obtain
\[
    P_w(a,q)
    =
    \frac{1}{(q-q^{-1})^n}
    \sum_{j=1}^n
        (-1)^{n-j}
        [s_{(j,1^{n-j})}]\LLT_w
        \bigl(a^{2j-n}-a^{2j-n-2}\bigr).
\]
The coefficient of $a^{2k-n}$ receives contributions only from
$j=k$ and $j=k+1$. Applying Theorem~\ref{thm:web-llt-hook} gives
\eqref{eq:homfly-hook-formula}.
\end{proof}

Theorem~\ref{thm:web-llt-hook} gives manifestly positive formulas for
the hook Schur coefficients of $[t(I)]!\,\LLT_w$. We conjecture
that the restriction to hook shapes can be removed.

\begin{conjecture}
\label{conj:web-llt-schur-positivity}
Let $w$ be an annular web, and let $I$ be a radial cut with
initial thickness $t(I)$. Then
$[t(I)]!\,\LLT_w$ is Schur-positive.
\end{conjecture}

\section{Comparison with chromatic symmetric functions for graphs}
\label{sec.comparison}

In this section, we compare annular webs with unit interval graphs.
We first realize unit interval graphs as annular webs.
We then express the chromatic symmetric function of an arbitrary
annular web as a weighted sum of graph chromatic quasisymmetric
functions indexed by its geometric loop decompositions.

Let $G$ be a finite loopless directed multigraph with vertex set $[n]$
and edge multiset $E(G)$. A \emph{proper coloring} is a map
$c:[n]\to\mathbb Z_{>0}$ such that $c(u)\neq c(v)$ for every
$(u,v)\in E(G)$. Write $\mathcal C(G)$ for the set of proper colorings
and define the \emph{(signed) $G$-inversion}\footnote{This agrees with \eqref{eq:def of inv}. The standard convention is recovered by the normalization $q^{|E(G)|/2}X_G[\mathbf x;q^{1/2}]$.} by
\[
    \inv_G(c)
    :=\sum_{(u,v)\in E(G)}\operatorname{sgn}(c(v)-c(u)),
\]
counting edges with multiplicity. In this convention, the
\emph{chromatic quasisymmetric function} of $G$ \cite{Sta95, SW16} is
\[
    X_G[\mathbf x;q]
    :=\sum_{c\in\mathcal C(G)}q^{\inv_G(c)}
      \prod_{v\in[n]}x_{c(v)}.
\]

A \emph{Dyck path} $\pi$ of size $n$ runs from $(0,0)$ to $(n,n)$ by unit
up-steps and right-steps, staying weakly above the diagonal. If $\pi(i)$
is the height of its $i$-th column, then
\[
    i\leq\pi(i)\leq n,
    \qquad
    \pi(1)\leq\cdots\leq\pi(n).
\]
The associated graph $G(\pi)$ has vertex set $[n]$ and edge set
\[
    E(G(\pi))=\{(i,j):i<j\leq\pi(i)\},
\]
corresponding to the cells between $\pi$ and the diagonal.
These are \emph{unit interval graphs}, also realizable as intersection
graphs of unit-length intervals, and $X_{G(\pi)}$ is symmetric
\cite{SW16}. Figure~\ref{fig:Dyck path} shows an example with
$(\pi(1),\ldots,\pi(5))=(3,4,5,5,5)$.

\begin{figure}[htbp]
  \centering

\begin{tikzpicture}[scale=0.5]

  \foreach \i in {0,1,...,5}
  \draw[color=gray!70] (\i,0) -- (\i,5);

  \foreach \j in {0,1,...,5}
  \draw[color=gray!70]  (0,\j) -- (5,\j);

  \draw (0,0) -- (5,0);
  \draw (0,0) -- (0,5);
  \draw[color=gray!70] (0,0) -- (5,5);

  \draw[ultra thick]
  (0,0) -- (0,3) -- (1,3) -- (1,4) -- (2,4) -- (2,5) -- (5,5);
\end{tikzpicture}
\qquad \qquad
\begin{tikzpicture}[scale=0.9]
  \foreach \i in {1,...,5}
  \filldraw (\i,2) circle (1.5pt);
  \foreach \i in {1,...,5}
  \node at (\i,1.6) {\i};
  \draw[thick]
  (1,2) -- (5,2);
  \draw[thick]
  (5,2) arc [start angle=0, end angle=180, radius=1];
  \draw[thick]
  (4,2) arc [start angle=0, end angle=180, radius=1];
  \draw[thick]
  (3,2) arc [start angle=0, end angle=180, radius=1];

\end{tikzpicture}

\caption{A Dyck path $\pi$ and the associated unit interval graph $G(\pi)$.}
\label{fig:Dyck path}
\end{figure}
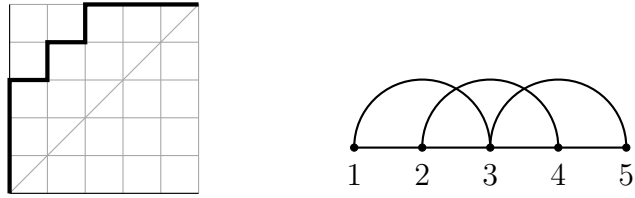

We now construct an annular web $w(\pi)$ associated to a Dyck path $\pi$. Start with $n$ upward-oriented unit strands
marked $1,\ldots,n$ from left to right. In what follows, a strand marked by a consecutive
block has thickness equal to the size of that block.

Set $\pi(0)=1$ and $b_i=\max\{i,\pi(i-1)\}$. At stage $i$, the active
strand is marked $\{i,\ldots,b_i\}$. All other strands are unit strands;
those marked $1,\ldots,i-1$ remain fixed in all later stages.
Writing $\mid$ for adjacent strands in left-to-right order, define the
merge and split operations by
\[
    \begin{aligned}
        M(j):&\quad
        \{i,\ldots,j-1\}\mid\{j\}
        &&\longmapsto \{i,\ldots,j\},\\
        S(i):&\quad
        \{i,\ldots,\pi(i)\}
        &&\longmapsto \{i\}\mid\{i+1,\ldots,\pi(i)\}.
    \end{aligned}
\]
Set $S(i)=\mathrm{id}$ when $\pi(i)=i$. Stage $i$ applies
\[
    W_i=S(i)M(\pi(i))\cdots M(b_i+1),
\]
with operators acting from right to left and the merge product empty
when $\pi(i)=b_i$. After $W_n\cdots W_1$, all strands again have
thickness one. Forgetting the markings and joining corresponding top
and bottom endpoints in the annulus defines $w(\pi)$. 
Figure~\ref{fig:dyck-path-associated-web}
provides examples.
\begin{figure}[!htbp]
	\centering
\begingroup%
  \makeatletter%
  \providecommand\color[2][]{%
    \errmessage{(Inkscape) Color is used for the text in Inkscape, but the package 'color.sty' is not loaded}%
    \renewcommand\color[2][]{}%
  }%
  \providecommand\transparent[1]{%
    \errmessage{(Inkscape) Transparency is used (non-zero) for the text in Inkscape, but the package 'transparent.sty' is not loaded}%
    \renewcommand\transparent[1]{}%
  }%
  \providecommand\rotatebox[2]{#2}%
  \newcommand*\fsize{\dimexpr\f@size pt\relax}%
  \newcommand*\lineheight[1]{\fontsize{\fsize}{#1\fsize}\selectfont}%
  \ifx\svgwidth\undefined%
    \setlength{\unitlength}{355.39043517bp}%
    \ifx\svgscale\undefined%
      \relax%
    \else%
      \setlength{\unitlength}{\unitlength * \real{\svgscale}}%
    \fi%
  \else%
    \setlength{\unitlength}{\svgwidth}%
  \fi%
  \global\let\svgwidth\undefined%
  \global\let\svgscale\undefined%
  \makeatother%
  \begin{picture}(1,0.31543149)%
    \lineheight{1}%
    \setlength\tabcolsep{0pt}%
    \put(0,0){\includegraphics[width=\unitlength,page=1]{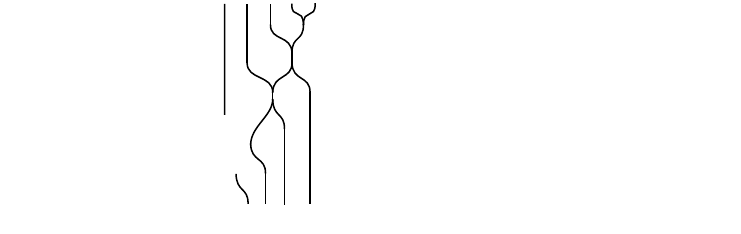}}%
    \put(0.09807088,0.00851508){\color[rgb]{0,0,0}\makebox(0,0)[lt]{\lineheight{1.25}\smash{\begin{tabular}[t]{l}$\pi$\end{tabular}}}}%
    \put(0.34235613,0.00576981){\color[rgb]{0,0,0}\makebox(0,0)[lt]{\lineheight{1.25}\smash{\begin{tabular}[t]{l}$w(\pi)$\end{tabular}}}}%
    \put(0.7011919,0.00509452){\color[rgb]{0,0,0}\makebox(0,0)[lt]{\lineheight{1.25}\smash{\begin{tabular}[t]{l}$\pi'$\end{tabular}}}}%
    \put(0.92279074,0.00604249){\color[rgb]{0,0,0}\makebox(0,0)[lt]{\lineheight{1.25}\smash{\begin{tabular}[t]{l}$w(\pi')$\end{tabular}}}}%
    \put(0,0){\includegraphics[width=\unitlength,page=2]{dyck.pdf}}%
  \end{picture}%
\endgroup%

	\caption{Dyck paths and their associated webs.}
	\label{fig:dyck-path-associated-web}
\end{figure}

The markings determine its unique loop decomposition
$(L_1,\ldots,L_n)$. Indeed, the strand separated at stage $i$ ends at
boundary position $i$ and is never used again. It must therefore belong
to the loop starting at position $i$. This forces every split
successively; the merges are determined by their incoming labels.

For $i<j$, the loop $L_j$ joins the active strand before $L_i$ is
separated exactly when $j\leq\pi(i)$. The two loops enter different
inputs only at that merge, with $L_i$ on the left and $L_j$ on the
right. Thus the directed merge graph is $G(\pi)$, and assigning colors
to $L_1,\ldots,L_n$ identifies proper colorings with the same weights.
Hence
\[
    X_{w(\pi)}[\mathbf x;q]=X_{G(\pi)}[\mathbf x;q].
\]

For a general annular web $w$ of degree $n$,
let $\mathbb L(w)$ be the set of geometric loop
decompositions, regarded as multisets. For each $L\in\mathbb L(w)$,
choose an auxiliary indexing $L_1,\ldots,L_n$ of its loop copies.
We denote by $\mu(L) \vdash n$ the partition of multiplicities of the distinct geometric loops,
and define a directed multigraph $G(L)$ on $[n]$ by adding
\begin{enumerate}
    \item an edge $(u,v)$ at each merge where $L_u$ enters from the left
          and $L_v$ from the right;
    \item an edge $(u,v)$ for each $u<v$ with $L_u=L_v$ as oriented
          edge-subgraphs.
\end{enumerate}
Edges are counted with multiplicity. The second rule gives a transitively
oriented complete graph on each class of coincident loop copies.

The underlying simple graph is the intersection graph of the loop
copies. Indeed, tracing distinct loops backwards from a shared edge
reaches a merge with different inputs. Loops meeting at a trivalent
vertex also share an edge, and the second rule accounts for coincident
copies.

\begin{proposition}
\label{prop:loop-decomposition-graph-expansion}
For every annular web $w$,
\[
    X_w[\mathbf x;q]
    =\sum_{L\in\mathbb L(w)}
        \frac{X_{G(L)}[\mathbf x;q]}{[\mu(L)]!}.
\]
\end{proposition}

\begin{proof}
A proper coloring of $G(L)$ induces a web coloring with loop
decomposition $L$. Conversely, every web coloring determines its loop
decomposition through its monochromatic components.

Fix a web coloring $c$ with decomposition $L$. Its lifts to $G(L)$
differ only by permuting colors among coincident loop copies. Their
merge contribution is always $\inv_w(c)$, while a class of $m$
coincident loops contributes
\[
    \sum_{\rho\in S_m}
        q^{\binom m2-2\operatorname{inv}(\rho)}=[m]!.
\]
The classes are independent, so summing over the lifts $\widetilde c$
of $c$ gives
\[
    \sum_{\widetilde c\mapsto c}
        q^{\inv_{G(L)}(\widetilde c)}
        \prod_{u=1}^n x_{\widetilde c(u)}
    =[\mu(L)]!\,q^{\inv_w(c)}x^c.
\]
Dividing by $[\mu(L)]!$ and summing over all web colorings proves the
identity.
\end{proof}

The running example in Figures~\ref{fig.SW} and~\ref{fig.noninv} has
four loop decompositions $L^{(1)},\ldots,L^{(4)}$, shown below.
The individual summands need not be symmetric. For example,
\[
    \begin{aligned}
        [x^{(2,1,1)}]X_{G(L^{(1)})}
        &=q^3+2q+3q^{-1},\\
        [x^{(1,2,1)}]X_{G(L^{(1)})}
        &=q^3+2q+2q^{-1}+q^{-3}.
    \end{aligned}
\]
At $q=1$, they may still fail to be Schur positive: both
$G(L^{(3)})$ and $G(L^{(4)})$ are claws, with
\[
    X_{G(L^{(3)})}[\mathbf x;1]
    =X_{G(L^{(4)})}[\mathbf x;1]
    =s_{31}-s_{22}+5s_{211}+8s_{1111}.
\]
Thus the Stanley--Stembridge theorem cannot be applied directly to
these summands to prove Conjecture~\ref{conj:e-positivity}.
Nevertheless, the sum of all four terms is the $e$-positive function
$X_w$.
\begin{figure}[!htbp]
	\centering
\begingroup%
  \makeatletter%
  \providecommand\color[2][]{%
    \errmessage{(Inkscape) Color is used for the text in Inkscape, but the package 'color.sty' is not loaded}%
    \renewcommand\color[2][]{}%
  }%
  \providecommand\transparent[1]{%
    \errmessage{(Inkscape) Transparency is used (non-zero) for the text in Inkscape, but the package 'transparent.sty' is not loaded}%
    \renewcommand\transparent[1]{}%
  }%
  \providecommand\rotatebox[2]{#2}%
  \newcommand*\fsize{\dimexpr\f@size pt\relax}%
  \newcommand*\lineheight[1]{\fontsize{\fsize}{#1\fsize}\selectfont}%
  \ifx\svgwidth\undefined%
    \setlength{\unitlength}{377.62868806bp}%
    \ifx\svgscale\undefined%
      \relax%
    \else%
      \setlength{\unitlength}{\unitlength * \real{\svgscale}}%
    \fi%
  \else%
    \setlength{\unitlength}{\svgwidth}%
  \fi%
  \global\let\svgwidth\undefined%
  \global\let\svgscale\undefined%
  \makeatother%
  \begin{picture}(1,0.3232797)%
    \lineheight{1}%
    \setlength\tabcolsep{0pt}%
    \put(0,0){\includegraphics[width=\unitlength,page=1]{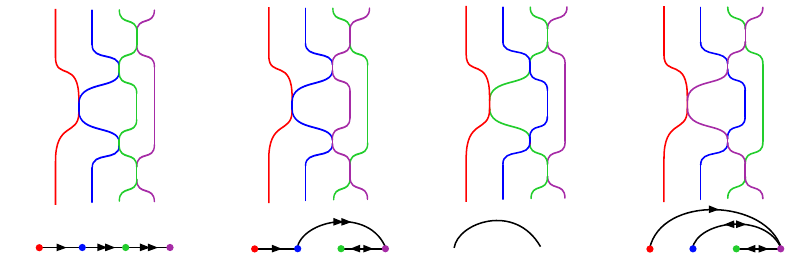}}%
    \put(0.00174761,0.17222958){\color[rgb]{0,0,0}\makebox(0,0)[lt]{\lineheight{1.25}\smash{\begin{tabular}[t]{l}$L^{(i)}$\end{tabular}}}}%
    \put(-0.08080638,0.00380667){\color[rgb]{0,0,0}\makebox(0,0)[lt]{\lineheight{1.25}\smash{\begin{tabular}[t]{l}$G(L^{(i)})$\end{tabular}}}}%
    \put(0,0){\includegraphics[width=\unitlength,page=2]{LoopDec.pdf}}%
  \end{picture}%
\endgroup%

\end{figure}

A similar phenomenon appears in \cite[Theorem~7.1]{ASS26}: an
$e$-positive sum whose summands need not be $e$-positive. The authors
ask for other interesting examples, and annular web chromatic
symmetric functions conjecturally provide such a family.

\subsection*{Acknowledgment}
J. Oh was supported by the National Research Foundation of Korea (NRF) grant funded by the Ministry of Education (RS-2026-25469155). S. Yoon was supported by the National Research Foundation of Korea (NRF) grant funded by the Ministry of Education (RS-2026-25478096). 

\subsection*{Declaration of AI use}
The authors used ChatGPT (OpenAI) for language editing, literature searches, and checking mathematical arguments and proofs for possible errors. It was not used to generate the mathematical ideas or results of this work. The authors independently verified all mathematical content and take full responsibility for the manuscript.

\bibliographystyle{hamsalpha}
\bibliography{biblio}
    
\end{document}